\documentclass[12pt]{puthesis_single_space}

\usepackage[percent]{overpic}
\usepackage{epsfig}

\usepackage{amsfonts}

\usepackage{amssymb}

\usepackage{amsmath}   

\usepackage{amsthm}

\usepackage{latexsym}

\usepackage{subfigure}

\usepackage{float}

\usepackage{enumerate}

\newtheorem{theorem}{Theorem}[section]

\newtheorem{definition}[theorem]{Definition}

\newtheorem{remark}[theorem]{Remark}

\newcommand{\disp}{\displaystyle}
\newcommand{\mbf}[1]{\mathbf{#1}}
\newcommand{\opnm}{\operatorname}
\newcommand{\beqn}{\begin{equation}}
\newcommand{\eeqn}{\end{equation}}

\newcommand{\tran}[1]{{#1}^{\operatorname{T}}}

\newcommand{\mc}{\mathcal}
\newcommand{\mbb}{\mathbb}
\newcommand{\rst}[1]{\ensuremath{{\mathbin\upharpoonright}\raise-.5ex\hbox{$#1$}}} 

\def\Xint#1{\mathchoice
{\XXint\displaystyle\textstyle{#1}}%
{\XXint\textstyle\scriptstyle{#1}}%
{\XXint\scriptstyle\scriptscriptstyle{#1}}%
{\XXint\scriptscriptstyle\scriptscriptstyle{#1}}%
\!\int}
\def\XXint#1#2#3{{\setbox0=\hbox{$#1{#2#3}{\int}$ }
\vcenter{\hbox{$#2#3$ }}\kern-.6\wd0}}

\def\dashint{\Xint-}

\title{Randomized Wavelets on Arbitrary Domains and Applications to Functional MRI Analysis}
\submitted{September 2012}  

\author{S. G\"orkem \"Ozkaya}

\advisor{Ingrid Daubechies}
\abstract{
Wavelets have been shown to be effective bases for many classes of natural
signals and images.  Standard wavelet bases have the entire vector space $\mathbb R^n$ as
their natural domain. It is fairly straightforward to adapt these to rectangular subdomains, and there also exist constructions for domains with more complex boundaries.   However those methods are ineffective when we deal with domains that
are very arbitrary and convoluted.  A particular example of interest is the human cortex, which is the part of the
human brain where all the cognitive activity takes place.  In this thesis, we use the lifting scheme to design wavelets on arbitrary volumes, and in particular on volumes having the structure of the human cortex.  These
wavelets have an element of randomness in their construction, which allows us to repeat
the analysis with many different realizations of the wavelet bases and
averaging the results, a method that improves the power of the analysis. 
Next, we apply this type of wavelet transforms to the statistical analysis to fMRI data, and we show that it enables us to achieve greater spatial localization than other, more standard techniques. 
}

\acknowledgements{
I am grateful to my advisor Ingrid Daubechies for the guidance, support,  and freedom that she provided to me during my graduate studies.  I have enormous respect and admiration  for her academic personality, her philosophy of applied mathematics,  and her human side.  I have been extremely fortunate to be her student.

I am thankful to  Dimitri Van De Ville for having me as a visiting graduate student at EPFL during the summer of 2010, and for his continuous mentorship from then on.  He has been a source of inspiration for me with his passion in science and positive attitude. 

I would like to thank the faculty at Princeton, from whom I learned a lot.   I would like to thank Phil Holmes and Amit Singer for participating in my dissertation committee, and Robert Calderbank, Weinan E, Peter Ramadge and  Sergio Verdu  for participating in my preliminary and general exams. 

I would like to thank Rayan Saab for giving his time generously for my questions and for his feedback during the writing of this dissertation. 

I would like to thank the professors who taught me during my undergraduate and Masters degrees at Bogazici University.  In particular, I would like to thank Alp Eden for all his support and encouragement while supervising my Masters thesis, and to him, Aysin Ertuzun, Nilgun Isik, Kadri Ozcaldiran and Ayse Soysal for their support as I applied to graduate school.

I would like to thank Howard Bergman, Christina Lipsky, Valerie Marino, Rebecca Louie and Audrey Mainzer for their help in administrative matters throughout my graduate studies. 

I would like to thank my friends Mihai Cucuringu, Aytug Goksu, Deniz Gunduz, Hakan Kaya, Taner Kilincoglu, Menachem Lazar, Onur Ozyesil, Gungor Polatkan, Jesus Puente, Fethi Ramazanoglu, Serdar Selamet, Soner Sevinc, Sergey Voronin,   and all my fellow graduate students for the friendly environment during the graduate school years.

I would like to thank my mother, my father, my brother, my parents-in-law and all
of my relatives in Turkey.  Their unconditional love and support throughout
the years has been the motivation that kept me working.

Finally, I would like to thank my wife Asl\i han Demirkaya for her patience, love, support and for always standing beside me throughout this dissertation.

}

\begin{document}

\chapter{Introduction}  
This thesis constructs and applies algorithms to carry out a wavelet analysis for data located in three-dimensional structures that are essentially ``fat surfaces'', i.e., are $\epsilon$-neighborhoods of possibly convoluted two-dimensional surfaces. This new approach was developed for the analysis of fMRI data.

This last paragraph contains several terms that we would like to explain already here, before going into the technical details of subsequent chapters. 
\subsubsection{Wavelets}
A \emph{wavelet basis} is a collection of functions $(\psi_{j,m})$, which typically have two indices $j$ and $m$, denoting \emph{scale} and \emph{position}, respectively.   In the classic setting, elements of a wavelet basis are translated and dilated versions of a prototype  called the {``mother wavelet''} function.  In other words, given a mother wavelet $\psi(x)$, the basis element $\psi_{j,m}$ is given by
$$
\psi_{j,m}(x) = C_j\psi(2^j(x-m)),
$$
where $j \in \mathbb Z$,  $m \in 2^{-j}\mathbb Z$  and $C_j$ is the normalizing constant that ensures $\|\psi_{j,m}\|=1$. 

It is generally expected that the mother wavelet has some degree of smoothness, and is either compactly supported, or has most of its energy  contained in a compact domain. 

 With a good wavelet basis, only a small fraction of the coefficients in the wavelet expansion
\[
\label{wavelet_expansion}
  f = \sum_{j,m}\gamma_{j,m}\psi_{j,m},
\]
suffices to approximate the function $f$ whenever it belongs to a  signal class of interest,  and a local change in the signal only affects a small number of corresponding coefficients.  These properties make wavelets optimal bases in compression, denoising and estimation applications \cite{ donoho93, daubechies,sweldens_siam}. 
 
 \subsubsection{Wavelets in higher dimensions}

  One dimensional wavelet designs can be translated into higher dimensional Euclidean spaces by means of tensor products, and these are called \emph{separable wavelet bases}. There are other means of obtaining wavelets directly in higher dimensions, typically leading to \emph{nonseparable wavelet bases}.  The Fourier transform was used as an indispensable tool in most of the constructions in the earlier years of the development of wavelet theory: most of the classical wavelet bases for $\mathbb R^n$, whether separable or nonseparable, depend on the Fourier transform in their designs. 
  
    \subsubsection{The lifting scheme and the second generation wavelets}
  Standard wavelet bases mentioned above have the entire vector space $\mathbb R^n$ as
their natural domain.   It is fairly straightforward to adapt these to rectangular subdomains, and there also exist constructions for domains with more complex boundaries.   However, those methods are ineffective when we deal with domains that
are very arbitrary and convoluted.  A particular example of interest is the cortex of the
human brain, which is the place where all the cognitive activity takes place.  
 For such  arbitrary subsets of $\mathbb R^n$, which do not possess a mathematical group structure, tools like the Fourier transform are not available to be utilized in the design.  Fortunately however, very flexible frameworks such as the the lifting scheme \cite{sweldens96, sweldens_siam} are available to construct wavelet bases on these domains directly in the spatial domain.  The wavelets obtained by such methods are no longer the translated and dilated copies of a prototype, and they are called the \emph{second generation wavelets}.  The second generation wavelets possess many of the appealing features of the first generation wavelet transforms, such as leading to sparse and localized representations, and the availability of fast transform algorithms. 

\subsubsection{Some prior applications of second generation wavelets}
One of the first constructions of biorthogonal wavelets on triangulated surfaces was by Launsbery et al. \cite{lounsbery94}.  This scheme was later improved by means of the lifting scheme, which was published by Sweldens \cite{mallat09, sweldens96, sweldens_siam}. 

Lifting-based wavelets on nonstandard grids have been used in numerous areas; an example is the construction in \cite{sweldens95}, of wavelet bases on the sphere, using an adaptive subdivision scheme.  Spherical wavelets have applications in astronomical and geophysical data analysis. 

Other examples of lifting-based wavelets have applications in computer graphics, where they can be used to efficiently represent meshes. In  \cite{khodakovsky04}, they are used in compression of three dimensional meshes. In \cite{wei01}, an application in texture synthesis over arbitrary manifolds is given. 
 
In \cite{cohen92, dahmen97, vasilyev00, griebel07}  wavelets of this type are used for numerical solutions of partial differential equations.  

An application in computational digital photography is given in \cite{fattal09}; in this paper a wavelet basis that is customized to a given image is constructed by the lifting scheme, which has wavelets and scaling functions that avoid edges, as much as possible.   These wavelets are then successfully applied to problems of dynamic range compression, edge preserving smoothing, detail enhancement and image colorization. 

A neuroimaging application is given in \cite{Yuetal07}, where it is applied to capturing shape variations of the cortex and in tracking cortical folding changes in newborns. 

For a review of the theory of second generation wavelets and a survey of other applications, we refer to \cite{jansen2005second}.
\subsubsection{Functional brain imaging}
Functional brain imaging refers to acquiring time series of three-dimensional images of the human brain, which measure some correlate of the neural activity.  This is different from anatomical (or structural) imaging in which the measured quantity is chosen to obtain as much contrast as possible to differentiate between different tissue types.

Functional magnetic resonance  imaging (fMRI) has become the most widely used functional brain imaging
technology in neuroscience research.  It has led to an explosion in papers aiming to
understand the organization of the human brain. Earlier
fMRI studies
mainly concentrated on determining the regions where the active component peaked, and were therefore not overly concerned about  losing lower-amplitude information before
applying spatial low pass
filtering to the data to improve the statistical power \cite{beyond_mind}.
However, later studies have demonstrated that fMRI signal
contains much additional detail information. For instance, the work by J. Haxby and his collaborators \cite{haxby} showed, in an experiment measuring discrimination between different stimuli, that enough information was present, even after removal of the main activation area for each stimulus, to allow for successful stimulus classification.  This
motivates the use of tools like wavelet analysis to detect 
fine details that might be less apparent and that could be lost to thresholding in the spatial domain.  However, the standard
wavelet transforms are mostly designed to have very special domains, like a rectangle in two dimensions or a cube in
three dimensions.  The activity of interest in the brain, on the other hand, occurs on the cortex, which is an intricately convoluted set
in three dimensions.   Using standard wavelets without any modification will have the obvious drawback of mixing signal from the cortex (gray matter) with the signal coming from the off-cortex regions.  There would be an artificial edge at the boundary of the cortex, which would inflate the detail coefficients in the wavelet transform, reducing its ability to form a sparse representation of the data. 

\subsubsection{Domain-adapted wavelets}
In this thesis, mainly motivated by the functional brain imaging problem, we use the lifting scheme to design wavelets on arbitrary two and three dimensional domains.  
One of the key steps in the construction of the wavelets  is to define a nested family of partitions on the domain.  This
includes parent, child, neighbor and sibling relations.  Our algorithm of defining these structures has an element of randomness in it, which results in getting multiple sets of wavelet bases, each coming from a different realization of the partitioning.  This in turn allows us to repeat
the analysis with many different realizations of the wavelet bases and
averaging the results, a method that improves the power of the analysis. 
An example of a nonstandard domain, a random partitioning on it and an example wavelet and scaling function (to be explained in the next chapter) is displayed in Figure \ref{fig:wavelet_and_scaling}.

\begin{figure}
\begin{center}
\subfigure[]{{\includegraphics[width=10cm]{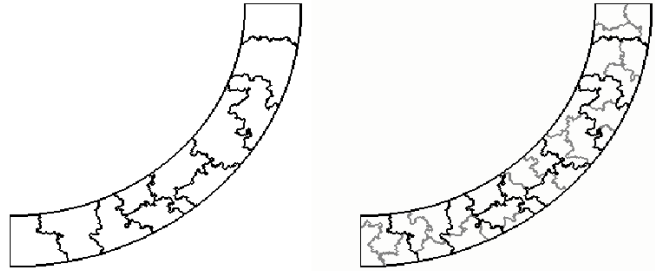}}}\\
\subfigure[]{{\includegraphics[width=7.5cm]{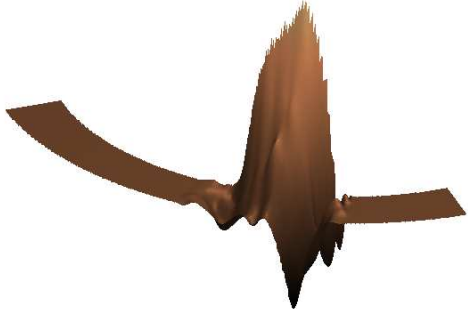}}}
\subfigure[]{{\includegraphics[width=7.5cm]{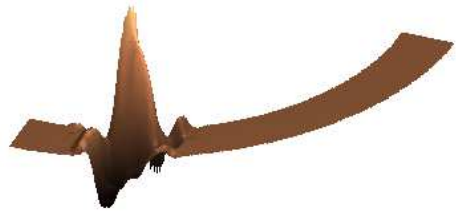}}}
\end{center}
\caption{Illustration of partitioning and adapted wavelets on an annular shaped domain. (a) Sample random partitioning of two levels. (b) Scaling function. (c) Wavelet}
\label{fig:wavelet_and_scaling}
\end{figure}

\subsubsection{Wavelet-based statistical parametric mapping}
One of the basic objectives in an fMRI study is to detect whether a given location point (voxel) is responding to the given stimulus or not.  The question can be answered by means of statistical tests that are applied to the corresponding time series.  However, these statistical tests always have the chance of a positive result when there was no response (false positive), and a negative result when there actually was some response (false negative).   In an fMRI image, there is a large number of voxels, and the statistical test needs to be performed for each of these voxels. This means that there is a large number of decisions to make about the absence or presence of response.  Due to the noisy nature of the fMRI data, some spatial filtering is necessary in order to have control over the total number of false positives, while not losing the sensitivity of the test altogether.  Statistical parametric mapping (SPM) \cite{hbf2}, which is one of the most widely used methods, uses a Fourier-based low pass filtering as a means of reducing noise.  However this has the drawback of destroying fine-scale spatial details.  An alternative is a  wavelet-based framework is due to Van De Ville et al. \cite{surfing, vandeville04,
vandeville0406, vandeville07}, which includes theoretical bounds on the false positive rate, when the null hypothesis is true in the region of interest.  This framework has no restrictions on the type of wavelet basis to be used, so it is suitable to use the domain adapted wavelets that are constructed so as to be spatially adapted to the cortex of the human brain. 

\subsubsection{Organization of the thesis}

The organization of the thesis will be as follows. In Chapter 2 we will study the preliminaries about multiresolution analysis, wavelet transforms, and the lifting scheme.  Chapter 3 will be about the methodology of construction of the adapted wavelets and  will also give some numerical experiments.  In Chapter 4 we will give results about the application of this type of wavelets in statistical analysis of fMRI data, in the wavelet-based statistical parametric mapping framework.  Chapter 5 contains a summary and conclusion. 
    
\chapter{Wavelets and the lifting scheme}  
In this chapter, we  review the basics of wavelet theory, as well as of the lifting scheme,  a method that can be used to design wavelets on irregular domains, without relying on tools from Fourier analysis. It is due to Wim Sweldens \cite{sweldens_siam, sweldens96}. 

Our presentation and notation will mostly follow \cite{sweldens_siam}. A generalized language for wavelets will be used, which also covers the classical (first generation) wavelets as a special case. 

\section{Wavelet bases and wavelet transforms}

In a general setting, a wavelet basis is a collection of functions 
\[
\label{wdef} 
 \left\{\psi_{j,m}(x)\colon j\in\mathcal J, m \in\mathcal M(j) \right\},
\]
where $\mathcal J\subset \mathbb Z$ is the index set for scale (frequency), and for each $j$, $\mathcal M(j)$ is the index set for location. The collection enables us to expand a given function $f$ as

\[
\label{wavelet_expansion}
  f = \sum_{j\in\mathcal J,\  m \in\mathcal M(j)}\gamma_{j,m}\psi_{j,m},
\]
which induces the representation of the function $f$ with a sequence of numbers $(\gamma_{j,m})$
\begin{align*}
f \mapsto (\gamma_{j,m}),
\end{align*}
  and this mapping is called a \emph{wavelet transform}.   
  
  In a good wavelet transform, only a small fraction of the coefficients suffices to approximate the function $f$ whenever it belongs to a  signal class of interest,  and a local change in the signal only affects a small number of corresponding coefficients.  These make wavelets optimal bases in compression, denoising and estimation applications \cite{sweldens_siam, donoho93, mallat09}.

\subsection{First generation wavelets}
In the classic setting, a wavelet basis is obtained by translating and dilating (up to a normalizing constant) a single template function, which is called the \emph{mother wavelet}.  These type of wavelets are referred as the \emph{first generation wavelets}. For the one-dimensional case, the domain is the infinite real line $\mathbb R$,  the scale index set is $\mathcal J = \mathbb N$, and for $j \in\mathcal J$, the location index set is $\mathcal M(j) = 2^{-j}\mathbb Z$.  Given a mother wavelet $\psi(x)$, the basis element $\psi_{j,m}$ is defined to be
$$
\psi_{j,m}(x) = C_j\psi(2^j(x-m)),
$$
where $C_j$ is the normalizing constant that ensures $\|\psi_{j,m}\|=1$. 

In order to obtain a desirable wavelet basis, the mother wavelet should either be compactly supported, or have most of its energy be contained in a compact domain. 

\subsection{Examples}
\begin{itemize}
\item The Haar  basis is the earliest known example of a wavelet basis. It is generated by the Haar mother wavelet function
 $$
  \psi(x) =   \left\{\begin{array}{rc}1 & \text{if } 0\leq x  < 1/2 \\
  -1 & \text{if } 1/2\leq x \leq 1\\
  0 & \text{elsewhere}\end{array}\ \ \ . \right.
  $$
 Haar wavelets form an orthonormal basis and they are compactly supported. Their main drawback is that they are not smooth, which results in wavelet representations that are not very sparse, contrary to what was desired. 

\item The Meyer wavelet is one of the first constructions of a smooth mother wavelet function that generates an \emph{orthogonal basis}, but it is not compactly supported, which is a drawback for computational purposes. 

\item The Daubechies wavelet family comes next chronologically, it consists of wavelet bases that are both smooth, compactly supported and orthogonal.

\item There are wavelets that form a \emph{biorthogonal} basis (defined in Subsection \ref{subsec:bases_review}), one of which is illustrated in Figure \ref{wavelets}, along with the above examples. 
\end{itemize}

 \begin{figure}
  \centering
  \includegraphics[width=15 cm]{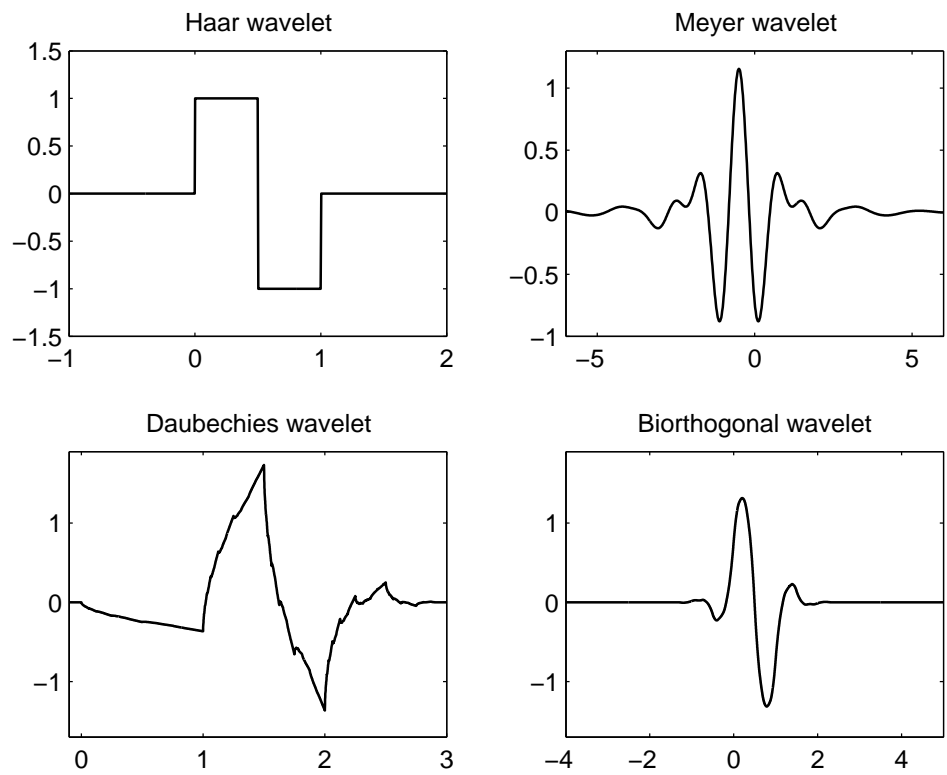}
  \caption{Examples of mother wavelets, belonging to some of the most commonly used wavelet bases: Haar, Meyer, Daubechies and Biorthogonal (average interpolating) wavelets. }
  \label{wavelets}
  \end{figure}
\section{Multiresolution analyses and fast wavelet transforms}
For the wavelet expansion 
\[
\label{wavelet_expansion}
  f = \sum_{j\in\mathcal J, \ m \in\mathcal M(j)}\gamma_{j,m}\psi_{j,m},
\]
orthogonal and biorthogonal wavelets have the relation 
\[
 \gamma_{j,m} = \langle f, \tilde \psi_{j,m} \rangle,
\]
that makes the computation of the coefficient functionals possible with a simple expression. However taking this inner product is computationally very slow.

Many of the wavelet bases, on the other hand,  are implicitly tied to a structure formed by a multilayered sequence of subspaces, which is called a \emph{multiresolution analysis}. 

A multiresolution analysis brings about a much faster way of computing the wavelet coefficients, which is called the \emph{fast wavelet transform} or the \emph{cascade algorithm}.

\subsection{Multiresolution analysis}
\label{subsec:multiresolution}
\begin{definition} A sequence of closed subspaces $\mbf M = \{V_j:j\in\mathcal J\subset \mathbb Z\}$ of $L^2(X,\Sigma,\mu)$  for some measurable space $X\subset\mathbb R^n$ is
  called a \emph{multiresolution analysis}  if it satisfies the following:
  \begin{enumerate}
  \item $V_j \subset V_{j+1}$
  \item $\bigcap_{j}V_j = \{0\}$
  \item $\overline{\bigcup_jV_j} = L^2(X,\Sigma,\mu)$
  \item For each $j\in\mathcal J$, there exists a Riesz basis $\{\varphi_{j,k}:j\in\mathcal K(j)\}$ that spans $V_j$. The functions \((\varphi_{j,k})\) are called \emph{scaling functions}. 
  \end{enumerate}
  \end{definition}
  Here, $\mathcal K(j)$ is an index set indicating location, and in general it is assumed that $\mathcal K(j) \subset \mathcal(K(j+1))$.  The scale index $\mathcal J$ is taken to be $\mathcal J = \mathbb Z$ when $X$ is an unbounded set of infinite measure,  and $\mathcal J= \mathbb N$ when $X$ is bounded and of finite measure (which is the case for the focus of this thesis).  

\subsubsection{Dual multiresolution analysis}
Let  $\mbf M = \{V_j:j\in\mathcal J\subset \mathbb Z\}$ be a multiresolution analysis. Then, another multiresolution analysis  
\[
\widetilde{\mbf M} = \{\widetilde V_j:j\in\mathcal J\subset \mathbb Z\}
\]
is  called a \emph{dual multiresolution analysis} of $\mbf M$, if its scaling functions $(\tilde \varphi_{j,k})$ satisfy 

\begin{equation}
\label{eqn:scaling_biorth}
\langle\varphi_{j,k},\tilde\varphi_{j,k'} \rangle = \delta_{k,k'} \ \ \ \text{ for } k,k'\in\mathcal K(j).
\end{equation}

For any given function $f\in V_j$, the coefficients in the expansion 
\[
 f = \sum_{k\in\mathcal K(j)}\lambda_{j,k} \varphi_{j,k},
\]
 can be obtained by
\[
 \lambda_{j,k} =\langle f , \tilde\varphi_{j,k} \rangle.
\]

\begin{remark}
Note that, according to this definition, a dual multiresolution is not unique.  We also note that, since $(\varphi_{j,k})$ is a Riesz basis for $V_j$, it also has a dual basis \emph{within} $V_j$ that is biorthogonal to it, as mentioned in Subsection \ref{subsec:bases_review}. However a set of dual scaling functions that are denoted by $(\tilde\varphi_{j,k})$ here are not necessarily contained in $V_j$, and they do not necessarily coincide with the \emph{dual Riesz basis} of $(\varphi_{j,k})$ within $V_j$.
\end{remark}

\subsection{Refinement equation, finite filters and set of partitionings}
\label{subsec:ref_eqn}
Because of the inclusion $V_j \subset V_{j+1}$,  $\varphi_{j,k}$ can be written as a linear combination of $(\varphi_{j+1,k})$'s:
\begin{equation}
\label{eqn:refinement}
 \varphi_{j,k} = \sum_{l\in\mathcal K(j+1)}h_{j,k,l}\varphi_{j+1,l},
\end{equation}
for any $j\in\mathcal J$ and $k\in\mathcal K(j)$. This is a relation that must be satisfied in any multiresolution analysis.  On the other hand, we shall see that a multiresolution analysis will be uniquely determined once the coefficients $(h_{j,k,l})$ are defined.  In order for this to be possible, we need two definitions. One is the concept of a \emph{finite filter}, which puts some constraints on the coefficients $(h_{j,k,l})$ for them to result in a well defined multiresolution analysis.  The other is the \emph{set of partitionings}, using which one will be able to pass from the finite filter to scaling functions in $L^2(X,\Sigma, \mu)$. 

\subsubsection{Finite filters}
\begin{definition}
 A set of real numbers $\{h_{j,k,l}\colon j\in\mathcal J,k\in\mathcal K(j), l\in\mathcal K(j+1)\}$ is called a \emph{finite filter} if the following are satisfied.
 \begin{enumerate}
  \item For each $j$ and $k$, $h_{j,k,l}$ is nonzero for only finitely many $l$'s.  Hence, the set defined as
 \begin{equation*}
    \mathcal L(j,k) = \{l\in\mathcal K(j+1) |h_{j,k,l}\neq 0\}
  \end{equation*}
is finite. 
    \item For each $j$ and $l$, $h_{j,k,l}$ is nonzero for only finitely many $k$'s.  Hence, the set defined as
  \[
   \mathcal K(j,l) = \{k\in\mathcal K(j) |h_{j,k,l}\neq 0\}
  \]
is finite. 
\item
The sizes of sets $\mathcal L(j,k)$ and $\mathcal L(j,l)$ are uniformly bounded over all $j,k$ and $l$.   
 \end{enumerate}
\end{definition}

\subsubsection{Dual filter}
We note that, for the dual scaling functions, there also exists a refinement relation
\begin{equation}
\label{eqn:dual_refinement}
 \tilde\varphi_{j,k} = \sum_{l\in\mathcal K(j+1)}\tilde h_{j,k,l}\tilde\varphi_{j+1,l},
\end{equation}
and using (\ref{eqn:scaling_biorth}), (\ref{eqn:refinement}) and (\ref{eqn:dual_refinement}), one obtains that the filter $(h_{j,k,l})$ and the dual filter $(\tilde h_{j,k,l})$ must satisfy the relation
\begin{equation}
\label{eq:h_biorth}
 \sum_{l} h_{j,k,l}\tilde h_{j,k',l} = \delta_{k,k'}, \ \  \text{ for } j\in\mathcal J, \ \ k,k'\in\mathcal K(j).
\end{equation}

\subsubsection{Nested set of partitionings}
In order to be able to define scaling functions on $L^2(X,\Sigma,\mu)$, the next structure that is needed is a \emph{set of partitionings.} In the following definition, we assume that $\mathcal K(j) \subset \mathcal K(j+1)$. 
\begin{definition}
\label{defn:partitionings}
 A \emph{nested set of partitionings} is a collection $\{S_{j,k}\colon j\in\mathcal J, k\in\mathcal K\}$ of subsets of $X$ that satisfy
 \begin{enumerate}
  \item For all $j \in \mathcal J$, the collection $\{S_{j,k}\colon k\in \mathcal K(j)\}$ is disjoint and $\displaystyle X=\overline{\bigcup_{k\in\mathcal K(j)}S_{j,k}}$,
  \item $S_{j+1,k}\subset S_{j,k}$,
  \item For all $j\in \mc J$, and $k\in \mc K(j+1)$, there exists a $k' \in \mc K(j)$ such that  $S_{j+1,k}\subset S_{j,k'}$, 
  \item For a fixed $k \in \mathcal K(j_0)$, the infinite intersection $\displaystyle\bigcap_{j>j_0}S_{j,k}$ is a single point set, whose unique element will be denoted by $x_k$. 
   \end{enumerate}
\end{definition}
\subsection{Synthesizing the scaling functions with the cascade algorithm}
\label{subsec:scaling_fun_cascade}
Having defined a set of partitionings $\{S_{j,k}\colon j\in\mathcal J, k\in\mathcal K(j)\}$, and a finite filter $\{h_{j,k,l}\colon j\in\mathcal J, k\in\mathcal K(j), l\in\mathcal K(j+1)\}$, one is now ready to synthesize the scaling function $\varphi_{j_0,k_0}$.  Initially, setting $(\lambda_{j_0,k})$ to a Kronecker sequence as
\[
 \lambda_{j_0,k} = \delta_{k,k_0}, \text{ for } k\in\mathcal K(j),
\]
we define the collection of sequences $(\lambda_{j,k})_{k\in \mathcal K(j)}$ for $j>j_0$, using the following recursive formula:
\[
 \lambda_{j+1,l} = \sum_{k\in\mathcal K(j,l)}h_{j,k,l}\lambda_{j,k}.
\]
Then a sequence of functions $\left(f^j_{j_0,k_0}\right)_j$ for $j\geq j_0$ is defined as
\[
 f^j_{j_0,k_0} = \sum_{k\in\mathcal K(j)}\lambda_{j,k}\chi_{S_{j,k}},
\]
which also satisfies
\begin{equation}
 \label{f_refinement}
 f^j_{j_0,k_0} = \sum_{l\in \mathcal K(j_0+1)}h_{j_0,k_0,l}f^j_{j_0+1,l}.
\end{equation}
Now the function $\varphi_{j_0,k_0}$ is defined to be
\[
 \varphi_{j_0,k_0}=\lim_{j\to\infty} f^j_{j_0,k_0},
\]
assuming the limit exists in $L^2$. 
\subsection{Wavelets}
\label{subsec:wavelets}
Given a multiresolution analysis $\mbf M = \{V_j:j\in\mathcal J\subset \mathbb Z\}$, and a dual $\widetilde{\mbf M} = \{\widetilde V_j:j\in\mathcal J\subset \mathbb Z\}$, the \emph{wavelet subspace} $W_j$ is the complement of $V_j$ in $V_{j+1}$ which is orthogonal to $\widetilde V_j.$ The \emph{wavelets} $\psi_{j,m}$ are the functions that span $W_j$.
\begin{definition}
\label{defn:wavelets}
 Let $\mathcal M(j) = \mathcal K(j+1)\setminus \mathcal K(j)$ be the index set complementing $\mathcal K(j)$ in $\mathcal K(j+1)$, and let $\{\psi_{j,m}\colon j\in\mathcal J,  m\in\mathcal M(j)\}\subset V_{j+1}$ be a collection of functions, and $W_j$ be the closure of its span. Then the collection $\{\psi_{j,m}\colon m\in\mathcal M(j) \}$ is a set of \emph{wavelet} functions if the following are satisfied. 
 \begin{enumerate}
  \item  $V_{j}\bigcap W_j = \{0\}$ and $W_{j} \perp \widetilde V_{j}$. 
  \item If $\mathcal J = \mathbb Z$, the set $\{\psi_{j,m}\colon j\in\mathcal J, m\in\mathcal M(j) \}$ is a Riesz basis for $L^2$. 
  
  If $\mathcal J = \mathbb N$, the set $\{\psi_{j,m}\colon j\in\mathcal J, m\in\mathcal M(j) \}\bigcup \{\varphi_{0,k}/{\|\varphi_{0,k}\|}\colon k\in\mathcal K(0)\}$ is a Riesz basis for $L^2$.
   \end{enumerate}
\end{definition}
The Riesz bases mentioned above will be referred as \emph{wavelet bases} from now on. 
\subsubsection{Dual wavelets}
Given a wavelet basis $(\psi_{j,m})$, there exists a corresponding dual Riesz basis (as explained in Subsection \ref{subsec:bases_review}), whose elements will be denoted by $(\tilde \psi_{j,m})$, and for a given $j$, the closed span of $(\tilde \psi_{j,m})$ will be denoted by $\widetilde W_j$.  The dual wavelets satisfy
\[
 \langle \psi_{j,m},\tilde \psi_{j',m'} \rangle = \delta_{m,m'}\delta_{j,j'},
\]
by their definition.  The dual wavelet subspaces also satisfy $\widetilde W_j \bigcap \widetilde V_j = \{0\}$ and   $\widetilde W_j \perp V_j$.

The dual wavelets can also be defined by reversing the roles of $\mbf M$ and $\widetilde{\mbf M}$, in Definition \ref{defn:wavelets}.
\subsection{Sets of biorthogonal filters}
\label{subsec:biorthogonality}
The finite filters $(h_{j,k,l})$ and $(\tilde h_{j,k,l})$ were defined in Subsection \ref{subsec:ref_eqn}.  These filters take place in the refinement equations 
\[
 \varphi_{j,k} = \sum_{l\in\mathcal K(j+1)}h_{j,k,l}\varphi_{j+1,l},
\]
and 
\[
 \tilde\varphi_{j,k} = \sum_{l\in\mathcal K(j+1)}\tilde h_{j,k,l}\tilde \varphi_{j+1,l}.
\]

Now similarly, since $\psi_{j,m}\in V_{j+1}$ and $\tilde \psi_{j,m} \in \widetilde V_{j+1}$, we can define the filters $g_{j,m,l}$ and $\tilde g_{j,m,l}$ to be the filters that satisfy 
\[
 \psi_{j,m} = \sum_{l\in\mathcal K(j+1)}g_{j,m,l}\,\varphi_{j+1,l},
\]
and 
\[
 \tilde\psi_{j,m} = \sum_{l\in\mathcal K(j+1)}\tilde g_{j,k,l}\,\tilde \varphi_{j+1,l}.
\]

The functions $(\varphi_{j,k})$, $(\tilde \varphi_{j,k})$, $(\psi_{j,m})$ and $(\tilde\psi_{j,m})$ satisfy the biorthogonality relations 
\begin{equation}
 \label{eqn:func_biorth}
 \begin{array}{lcl}
 \langle \varphi_{j,k},\tilde \varphi_{j,k'} \rangle &=& \delta_{k,k'},\\
 \langle \psi_{j,m},\tilde \psi_{j,m} \rangle &=& \delta_{m,m'},\\
\langle \varphi_{j,k},\tilde \psi_{j,m} \rangle &=& 0,\\
\langle \psi_{j,m},\tilde \varphi_{j,k} \rangle &=& 0.\\
 \end{array}
\end{equation}
As a consequence of the above relations, the corresponding filters must satisfy

\begin{equation}
\label{eqn:biorthogonal_filters}
\begin{array}{rcl}
\disp \sum_{l\in\mathcal K(j+1)}\, h_{j,k,l}\tilde h_{j,k',l} & = & \delta_{k,k'},  \\
\disp \sum_{l\in\mathcal K(j+1)}\, g_{j,m,l}\tilde g_{j,m',l} &=& \delta_{m,m'} ,  \\
\disp \sum_{l\in\mathcal K(j+1)}\, g_{j,m,l}\tilde h_{j,k,l} &=& 0  ,\\
\disp \sum_{l\in\mathcal K(j+1)}\, h_{j,k,l}\tilde g_{j,m,l} &=& 0.
\end{array}
\end{equation}

\begin{definition}
 A set of finite filters $\{h,\tilde h, g, \tilde g\}$ is called a set of {\rm\emph{biorthogonal filters}}, if the relations in {\rm(\ref{eqn:biorthogonal_filters})} are satisfied. 
\end{definition}

\begin{remark}
 The collection of wavelet and scaling functions that satisfy {\rm\eqref{eqn:func_biorth}} results in finite filters that satisfy {\rm(\ref{eqn:biorthogonal_filters})}.  But also, thinking in the reverse direction, if one starts with a set of biorthogonal filters, one can obtain the wavelet and scaling functions using the cascade algorithm of Subsection {\rm\ref{subsec:scaling_fun_cascade}}, and they will have the desired biorthogonal relations of {\rm(\ref{eqn:func_biorth})}.  The purpose of the lifting scheme is actually to design a set of biorthogonal filters and make succesive improvements on them without disturbing their biorthogonality. 
\end{remark}

\subsection{Fast wavelet transforms}
The fast wavelet transform is an  efficient algorithm that finds the expansion coefficients in the $N$-level wavelet expansion  
\begin{equation}
\label{eqn:fast_wavelet}
f = \sum_{k\in\mathcal K(0)}\lambda_{0,k}\,\varphi_{0,k} + \sum_{j = 0}^{N-1}\sum_{m\in\mathcal M(j)} \gamma_{j,m}\, \psi_{j,m},
\end{equation}
for a positive integer $N$ and  $f \in V_{N}$, by exploiting the refinement relations.  

If we take a function $f$ in $V_N$, it can be written as a linear combination of the scaling functions of $V_N$ as
\begin{equation}
\label{eqn:finest_level_expansion}
 f = \sum_{k\in\mathcal K(N)}\lambda_{N,k}\,\varphi_{N,k}.
\end{equation}
At the beginning of a fast wavelet transform, it is assumed that the scaling function coefficients $(\lambda_{N,k})_k$ at this finest level  are given.  Since $W_{N-1}$ is a complement of $V_{N-1}$ in $V_{N}$, $f$ can be written as $f = f_a + f_d$ with $f_a \in V_{N-1}$ and $f_d\in W_{N-1}$.  If we expand $f_a$ and $f_d$ in terms of the scaling functions and wavelets, we get
\begin{align*}
 f &= \sum_{k\in\mathcal K(N)}\lambda_{N,k}\,\varphi_{N,k}\\
 &=  \sum_{k\in\mathcal K(N-1)}\lambda_{N-1,k}\,\varphi_{N-1,k} + \sum_{m\in\mathcal M(N-1)}\gamma_{N-1,m}\,\psi_{N-1,m}.
\end{align*}
The key observation at this point is that the coefficient sequences $(\lambda_{N-1,k})_k$ and $(\gamma_{N-1,m})_m$ can be obtained by applying filters $\tilde h$ and $\tilde g$  to $(\lambda_{N,k})_k$, respectively. That is to say, the relations
\begin{align*}
 \lambda_{N-1,k} = \sum_{l\in\mathcal K(N)}\tilde h_{N,k,l}\,\lambda_{N,l}
\end{align*}
and
\begin{align*}
 \gamma_{N-1,k} = \sum_{l\in\mathcal K(N)}\tilde g_{N,k,l}\,\lambda_{N,l}
\end{align*}
can be obtained as a result of the biorthogonality relations.  Recursive application of this relation to the $(\lambda_{N-i,k})$ after each step results in (\ref{eqn:fast_wavelet}).

The inverse transform of the above step is given by
\begin{align*}
 \lambda_{N,l} = \sum_{k\in\mathcal K(N-1)} h_{N-1,k,l}\,\lambda_{N-1,k}
+
 \sum_{m\in\mathcal M(N-1)} g_{N-1,m,l}\,\gamma_{N-1,m}.
\end{align*}

We recall that, in the present setting, it is assumed that all the summations in the above equations are finite, since the filters are assumed to be finite. 
\section{The lifting scheme}
The lifting scheme is the framework through which one can obtain new biorthogonal filters from old ones, and tailor them according to the desired smoothness and vanishing moment properties of the corresponding scaling functions and wavelets.  
\subsection{Operator notation}
Continuing to follow \cite{sweldens_siam}, we will introduce the operator notation corresponding to the filters of the previous section. The operator notation will allow a much compact expression for the lifting scheme. 

First we will give a general definition for an operator $H$ and its adjoint $H^*$, corresponding to a filter $(h_{k,l})$.  Then the definitions specific to present case will follow. 

Let $\mathcal I_1$ and $\mathcal I_2$ be two index sets that are at most countable, and let $h = \{h_{k,l}\colon k\in\mathcal I_2, l\in\mathcal I_1\}$ be a filter with real coefficients.  Let $\ell^2(\mathcal I_1)$ and $\ell^2(\mathcal I_2)$ be the spaces of the square summable, real  valued discrete functions, defined on the index sets.
The operator $H$ corresponding to the filter $(h_{k,l})$ is  be defined to be
\begin{align}
\label{eqn:operator}
  H\colon & \ell^2(\mathcal I_1) \rightarrow  \ell^2(\mathcal I_2) \notag
\\
& \left(a_{k}\right)_{m\in\mathcal I_1} 
\mapsto \left(\sum_{l\in\mathcal I_1} h_{k,l}a_{l}\right)_{k\in\mathcal I_2}.
\end{align}

The adjoint operator of $H$, which is denoted by $H^*$,  is an operator from $\ell^2(\mathcal I_2)$ to $\ell^2(\mathcal I_1)$ given by
\begin{align}
\label{eqn:operator_adjoint}
  H^*\colon & \ell^2(\mathcal I_2) \rightarrow  \ell^2(\mathcal I_1) \notag
\\
& \left(a_{k}\right)_{m\in\mathcal I_2} 
\mapsto \left(\sum_{k\in\mathcal I_2} h_{k,l}a_{k}\right)_{l\in\mathcal I_1}.
\end{align}

Given an operator $H$, the adjoint operator is also the unique operator that satisfies 
\[
 \langle x, H y \rangle = \langle H^* x, y \rangle
\]
for all $x \in \ell(\mathcal I_1)$ and $y\in\ell(\mathcal I_2)$.  

The above notation will be used for the rest of the section.  Note that, in  (\ref{eqn:operator}), the summation runs over the first index of $(h_{k,l})$, whereas in (\ref{eqn:operator_adjoint}), it runs over the second index. For the case of matrices, the adjoint operator corresponds to the transpose, and the order of indexes for the filters has been set accordingly.

Now, let us fix a level $j \in \mathcal J$.  We have three index sets $\mathcal K(j+1)$, $\mathcal K(j)$ and $\mathcal M(j)$, which satisfy
\[
 \mathcal K(j+1) = \mathcal K(j)\, \cup\, \mathcal M(j). 
\]

Let us consider the spaces $\ell^2\left(\mathcal K(j+1)\right)$, $\ell^2\left(\mathcal K(j)\right)$ and $\ell^2\left(\mathcal M(j)\right)$.  We define the  operators $\tilde H_j$, $\tilde G_j$, $H_j^*$ and $G_j^*$ as follows:
\begin{align*}
\tilde H_j\colon &\ell^2(\mathcal K(j+1)) \rightarrow  \ell^2(\mathcal K(j))
\\
&\left(a_k\right)_{k\in\mathcal K(j+1)} \mapsto \left(\sum_{l\in\mathcal K(j+1)}\tilde h_{j,k,l}\,a_l\right)_{k\in\mathcal K(j)},
\\
&\\
\tilde G_j\colon & \ell^2(\mathcal K(j+1)) \rightarrow  \ell^2(\mathcal M(j))
\\
& \left(a_k\right)_{k\in\mathcal K(j+1)} 
\mapsto 
\left(
  \sum_{l\in\mathcal K(j+1)}\tilde g_{j,m,l}a_{l}
\right)_{m\in\mathcal M(j)},
\end{align*}
and similarly,
\begin{align*}
 H^*_j\colon & \ell^2(\mathcal K(j)) \rightarrow  \ell^2(\mathcal K(j+1))
\\
& \left(a_{k}\right)_{k\in\mathcal K(j+1)} 
\mapsto 
\left(\sum_{l\in\mathcal K(j)} h_{j,k,l}a_{l}\right)_{k\in\mathcal K(j)}
,
\\
&\\
 G^*_j\colon & \ell^2(\mathcal M(j)) \rightarrow  \ell^2(\mathcal K(j+1))
\\
& \left(a_{m}\right)_{m\in\mathcal M(j+1)} 
\mapsto \left(\sum_{l\in\mathcal K(j)} g_{j,k,l}a_{l}\right)_{k\in\mathcal K(j)}
.
\end{align*}

\subsection{Fast wavelet transform and biorthogonality with the operator notation}
Let us denote the sequence $(\lambda_{j,k})_{k\in\mathcal K(j)}$ simply by $\lambda_j$, and similarly $(\gamma_{j,m})_{m\in\mathcal M(j)}$  by $\gamma_j$.  The fast wavelet transform is the process that starts with $\lambda_N$, and continues as
\begin{align*}
 \lambda _N &\mapsto (\lambda_{N-1},\gamma_{N-1})
 \\
 &\mapsto (\lambda_{N-2},\gamma_{N-2},\gamma_{N-1})
 \\
 & \ \ \vdots\\
 &\mapsto (\lambda_{0},\gamma_{0},\gamma_{1},\cdots,\gamma_{N-2},\gamma_{N-1}),
\end{align*}
where 
\[
 \lambda_{j} = \tilde H_{j}\lambda_{j+1} \ \ \text{ and } \ \ \ \gamma_{j} = \tilde G_{j}\gamma_{j+1},
\]
at each step.

The inverse of each step can be obtained by the operators $H^*$ and $G^*$ as
\[
\lambda_{j+1} = H_{j}^*\lambda_{j} + G_{j}^*\lambda_{j}.
\]

The biorthogonality relations (\ref{eqn:biorthogonal_filters}) in Subsection \ref{subsec:biorthogonality} can be written with the operator notation as
\begin{align}
\label{eqn:biorthogonality_op1}
 \tilde G_j H^*_j = \tilde H_j G^*_j = 0 \\
 \label{eqn:biorthogonality_op2}
  \tilde H_j H^*_j = \tilde G_j G^*_j = I,
\end{align}
where $I$ is the identity operator.  From these, it follows that
\[
 \tilde  H^*_j H_j + G^*_j \tilde G_j  = I,
\]
provided the union of the ranges of the operators $H^*_j$ and $G^*_j$ span entire $\ell^2(\mathcal(K(j+1))$.  In finite dimensions, it is enough that $|\mc K(j+1)| = |\mc K(j)| + |\mc M(j)|$, in order for this to be satisfied. 

\subsection{Obtaining a new biorthogonal set of filters from old ones}
\subsubsection{Lifting}
The key point of the lifting scheme is to obtain new biorthogonal filters from old ones without losing the biorthogonality.  
\begin{theorem}
\label{thm:lifting_primal}
 Let ${H^{\text{\rm old}}, \tilde H^{\text{\rm old}}, G^{\text{\rm old}}, \tilde G^{\text{\rm old}}}$ be a set of biorthogonal filter operators.  Then the following gives a new set of biorthogonal filter operators:
 \begin{align*}
  H_j &= H_j^{\,\text{\rm old}},\\
  \tilde H_j &= \tilde H_j^{\text{\rm old}} + S \tilde G_j^{\text{\rm old}},\\
  G_j &= G_j^{\text{\rm old}}- S^*H_j^{\text{\rm old}},\\
  \tilde G_j &= \tilde G_j^{\text{\rm old}},
   \end{align*}
   where $S$ is any operator from $\ell^2(\mc M(j))$ to $\ell^2(\mc K(j))$. 
\end{theorem}
\begin{proof}
 For  (\ref{eqn:biorthogonality_op1}) we have
 \begin{align*}
 \tilde G_j H^*_j &= \tilde G_j^{\,\text{old}} H_j^{\,\text{old}} = 0,\\
 \tilde H_j G^*_j 
 &= 
      \left( \tilde H_j^{\,\text{old}} + S \tilde G_j^{\,\text{old}}\right) \Big(G_j^{\,\text{old}}- S^*H_j^{\,\text{old}}\Big)^*\\
 &= 
      \left( \tilde H_j^{\,\text{old}} + S \tilde G_j^{\,\text{old}}\right) \Big(G_j^{* \, \text{old}}-  H_j^{* \, \text{old}}S\Big)\\
%
& = 
\tilde H_j^{\,\text{old}}  G_j^{* \, \text{old}}
-\tilde H_j^{\,\text{old}}  H_j^{* \, \text{old}}S
+S \tilde G_j^{\,\text{old}}G_j^{* \, \text{old}}
-S \tilde G_j^{\,\text{old}} H_j^{* \, \text{old}}S\\
&=0 -S + S +0= 0.
\end{align*}
Similarly, one can verify that (\ref{eqn:biorthogonality_op2}) also holds.  
\end{proof}
The modification step to the original operators is called a \emph{lifting step}.  This modification changes $\tilde H$ and $G^*$, but it does not change $H^*$ and $\tilde G$. This implies that, after such a modification, the scaling functions remain the same, but the dual scaling functions and the wavelets change. The dual wavelets also change, since the dual scaling functions change, although the coefficients $(\tilde g_{j,m})$ in their refinement relation stay the same.
\subsubsection{Dual lifting}
There also exists another version of Theorem \ref{thm:lifting_primal}, in which the filters $H^*$ and $\tilde G$ are modified and  $\tilde H$ and $G^*$ remain unchanged, which we call as the \emph{dual lifting}. 
\begin{theorem}
\label{thm:lifting_dual}
 Let ${H^{\text{\rm old}}, \tilde H^{\text{\rm old}}, G^{\text{\rm old}}, \tilde G^{\text{\rm old}}}$ be a set of biorthogonal filter operators.  Then the following gives a new set of biorthogonal filter operators:
 \begin{align*}
  H_j &= H_j^{\,\text{\rm old}} +  R^*G_j^{\text{\rm old}},\\
  \tilde H_j &= \tilde H_j^{\text{\rm old}},\\
  G_j &= G_j^{\text{\rm old}},\\
  \tilde G_j &= \tilde G_j^{\text{\rm old}} - R \tilde H_j^{\text{\rm old}},,
   \end{align*}
   where $R$ is any operator from $\ell^2(\mc K(j))$ to $\ell^2(\mc M(j))$. 
\end{theorem}

\subsection{Heuristics for a good discrete wavelet transform}
\label{subsec:heuristics}
Now, we can focus on a  {(single-level) discrete wavelet transform} $T_j$, which maps a given discrete function $\lambda_{j+1}\in\ell^2(\mathcal (K(j+1)))$  into two discrete functions $(\lambda_j, \gamma_j)$, where $\lambda_{j}\in\ell^2(\mathcal (K(j)))$ and $\gamma_{j}\in\ell^2(\mathcal (M(j)))$ as
\begin{align*}
T_j: \ell^2{\mc K(j+1)} &\rightarrow  \ell^2{(\mc K(j))} \times \ell^2{(\mc M(j))}\\
\lambda_{j+1} &\mapsto (\lambda_j, \gamma_j).
\end{align*}
Here, the discrete functions $\lambda_j$ and $\gamma_j$ are interpreted as approximation and detail functions, respectively.  

Now, we can list the properties generally expected from a wavelet transform. In the following, we will implicitly identify the discrete function 
 $(\lambda_j)$, with the corresponding function $f\in L^2(X,\Sigma,\mu)$ given by
\(
 f = \sum_{k}\lambda_{j,k}\varphi_{j,k},
\)
when talking about concepts like \emph{smoothness}. 
\begin{enumerate}[W1:]
 \item \label{wt1} If  $\lambda_{j+1}$ belongs to a class of smooth functions to be specified (most commonly, a class of polynomials up to a certain degree), then the detail function $\gamma_j$ is expected to be exactly zero.
 \item \label{wtapp} $\lambda_{j}$ is expected to be an approximation of $\lambda_{j+1}$, in the sense that setting $\gamma_j =0 $ and taking the inverse of $(\lambda_j, 0)$    under $T_j$ should give an approximation of $\lambda_{j+1}$. 
 \item \label{wtcont} The transform is expected to be stable, i.e., a small change in $\lambda_{j+1}$ should correspond to small changes in $\lambda_j$ and $\gamma_j$, both in the transform and its inverse.
 \item \label{wtlocal} The transform is expected to be local. Changes in a localized area of the function should only affect a corresponding localized region of the transformed functions. 
 \item \label{wtsparse}As a consequence of the above conditions, if $\lambda_{j+1}$ is locally well approximated by the class of smooth functions, then many entries of the detail output $\gamma_j$ will be close to zero. This is the reason for the \emph{sparsity} property of the wavelet transforms, a key point of their success. 
 \end{enumerate}

\subsection{Tailoring a wavelet transform in successive lifting steps}
In order to design a wavelet transform, one can start with a very simple transform, and in successive steps,  improve it to meet the guidelines of Subsection \ref{subsec:heuristics}.  In doing so, Theorems \ref{thm:lifting_primal} and \ref{thm:lifting_dual} will be utilized.

\subsubsection{The lazy wavelet transform}
The lazy wavelet transform simply splits the given signal $\lambda_{j+1}$ into two, i.e. $\lambda_j$ and $\gamma_j$ are just restrictions of $\lambda_{j+1}$ to $\mc K(j)$ and $\mc M(j)$ respectively,
\begin{gather*}
T_j^{\text{lazy}}: \ell^2{\mc K(j+1)} \rightarrow  \ell^2{(\mc K(j))} \times \ell^2{(\mc M(j))} \\
\lambda_{j+1} \mapsto (\lambda_{j+1}\rst{\mc K(j)}, \  \lambda_{j+1}\rst{\mc M(j)}).
\end{gather*}
In other words, the corresponding filters are given by $h^{\text{lazy}}_{j,k,l} = \delta_{k,l}$ and $g^{\text{lazy}}_{j,m,l} = \delta_{m,l}$.  
  We will denote the corresponding filter operators as $\tilde H_j^\text{lazy}$ and $\tilde G_j^\text{lazy}$, and use the ordered pair notation for the operator as 
\[
T_j^{\text{lazy}} = (\tilde H_j^\text{lazy} , \tilde G_j^\text{lazy}). 
\]
It is easy to see that the lazy wavelet transform does not satisfy property W\ref{wt1}. There is no reason to expect that the signal would have values close to zero on the complementary grid $\mc M(j)$, which are typically uniformly scattered through the initial grid $\mc K(j+1)$. However, the lazy wavelet transform is very helpful as an initial step in the lifting framework, to be followed by a sequence of  improvement steps.
\subsubsection{Prediction}
The first improvement to the lazy wavelet transform comes with the help of a prediction operator.  Let $\lambda^{(0)}_{j} = \tilde H_j^\text{lazy}\lambda_{j+1}$ and $\gamma^{(0)}_j = \tilde G_j^\text{lazy}\lambda_{j+1}$ be the outputs of the lazy wavelet transform. The output of the  prediction operator is a function of $\lambda^{(0)}_{j}$ and under the smoothness assumption, it approximates the signal $\gamma^{(0)}_{j}$
\begin{gather*}
P\colon  \ell^2(\mc K(j)) \rightarrow  \ell^2(\mc M(j))\\
P \lambda^{(0)}_{j} \approx \gamma^{(0)}_{j}.
\end{gather*}
The improved transform now becomes
\begin{gather*}
T^{(1)}_j\colon  \ell^2 ({\mc K(j+1)}) \rightarrow  \ell^2 ({\mc K(j)})\times  \ell^2({\mc M(j)})\\
\lambda_{j+1} \mapsto (\lambda^{(0)}_{j}, \gamma^{(0)}_{j} - P\lambda^{(0)}_{j})
\end{gather*}
If the smoothness assumptions are satisfied, and the prediction successfully satisfies  $P \lambda^{(0)}_{j} \approx \gamma^{(0)}_{j}$,  then the desired property W\ref{wt1}, follows immediately, as intended.  The new filter operators are $\tilde H^{(1)}_j = \tilde H^{lazy}_j$ and  $\tilde G^{(1)}_j = \tilde G^{lazy}_j - P H^{lazy}_j$.  The complete set of biorthogonal filters can be obtained by Theorem \ref{thm:lifting_dual}.

\subsubsection{Locality of the Prediction}
In order to have the property W\ref{wtlocal}, we must impose a condition on the prediction operator $P$: the $m$th entry of $P\lambda^{(0)}_{j}$ should depend only on the \emph{neighbors} of the $m$th element of the complimentary grid $\mc M(j)$, noting that $\mc K(j) \bigcup \mc M(j) = \mc K(j+1)$.   Without this condition, the transform would not be local, and would not possess the advantages of wavelet transforms. 

\subsubsection{The Update Step}

At this point, we have a  transform  $T^{(1)}$ that satisfies property W\ref{wt1}. However, there is something unsatisfying about the \emph{approximation output} $\lambda^{(1)}_{j}:= \tilde H^{(1)}_j \lambda_{j+1}$, which is still simply  a restriction of $\lambda_{j+1}$ to the subgrid $\mc K(j)$.  This constitutes a violation of property W\ref{wtapp}.  In order to see this, one can take an extreme example for $\lambda_{j+1}$, such as
\[       \lambda_{j+1,k} 
= \left\{\begin{array}{rcl}  
0 &\text{if}& k\in \mc K(j) \\
1 &\text{if}& k\in\mc M(j)
\end{array}  
,\right.  
\] 
whose reconstruction after setting the detail coefficients $\gamma_j$ to $0$ will result in the $0$ function, which is obviously not a good approximation for the initial $\lambda_{j+1}$. 
 As a result, a new step is required, which will ensure that the local averages of the function reconstructed with $\lambda_{j}$ will match those of $\lambda_{j+1}$.  It can be accomplished by means of an operator $U$ that is local,  and acts on the detail output $\gamma^{(1)}_{j}:= \tilde G^{(1)}_j \lambda_{j+1}$ of the previous step. The new operators will be $\tilde H^{(2)}_j : = \tilde H^{(1)}_j + U \tilde G^{(1)}_j$, and $\tilde G^{(2)}_j := \tilde G^{(1)}_j$. The condition that is desired to be satisfied by the output $\lambda^{(2)}_{j}$ of $\tilde H^{(2)}_j$ is to  preserve the weighted sum of the coefficients as

\[
\sum_{k\in\mc K(j+1)} \lambda_{j+1,k}\,\mu(S_{j+1,k})=\sum_{k\in\mc K(j)} \lambda^{(2)}_{j,k}\,\mu(S_{j,k}).
\]

The modifications to get a complete set of biorthogonal filters is given by Theorem \ref{thm:lifting_primal}.  Note also that the update step for the primal filters corresponds to a prediction step for the dual filters. 
\vspace{.6cm} 

\noindent There can be as many ``predict'' and ``update'' steps  as one desires, to be added with similar rules.  However, in general additional steps with improved predictions requires the use of filters with a higher number of nonzero coefficients.  This, in turn, results in less localized wavelets and scaling functions. 
\subsubsection{}
In this chapter, we gave an overview of wavelets and the lifting scheme, a flexible framework that enables  designing wavelets on unstructured domains.  This chapter was essentially a preparation for the next chapter, in which we give a concrete  construction that works for arbitrary subsets of $\mathbb R^2$ and $\mathbb R^3$.

\newpage

\section{Appendix: Orthogonal bases, Riesz bases and frames}
\label{subsec:bases_review}
\subsection{Riesz bases and orthogonal bases}
In a Hilbert space \(\mathcal H \), a  basis $(x_n)$ is called a \emph{Riesz basis} if there exist two constants $A$ and $B$ such that
\begin{equation}
\label{eqn:riesz_basis}
A\sum{|a_n|^2}\leq \left\| \sum{a_n x_n}\right\|^2 \leq  B\sum{|a_n|^2}, 
\end{equation}
for all scalar sequences $(a_n)$. $A$ and $B$ are called the Riesz basis constants. 

An \emph{orthogonal basis} is a special Riesz basis in which $A=B=1$.  In this case, the basis satisfies
\[
\langle x_n, x_{n'}\rangle = \delta_{n,n'},
\]
and for an $f\in\mathcal H$ , the scalar coefficients  $(a_n)$ in the basis expansion 
\[
 f = \sum_{n}a_n x_n
\]
 can be computed by taking inner products with the corresponding basis elements:
\[
\label{ortho_coef_functional}
a_n = \langle f,x_n\rangle.
\]
\subsection{Biorthogonal bases}
For each Riesz basis \(  (x_n)  \) that is not orthogonal, there exist another Riesz basis  \(  (\tilde x_n)  \) that satisfies
\begin{equation*}
  \label{eqn:biorthogonal}
\langle x_n, \tilde x_{n'}\rangle = \delta_{n,n'},
\end{equation*}
and therefore one has the relation
\[
a_n = \langle f,\tilde x_n\rangle,
\]
useful for the computation of the coefficients.  The bases \(  ( x_n)  \) and \(  (\tilde x_n)  \) are called \emph{biorthogonal} bases.  For the proof of existence of a biorthogonal basis for every Riesz basis, we refer to \cite{christensen08}.

\subsection{Frames}
A frame is a countable collection of vectors $(e_n)$ in a Hilbert space $\mathcal H$, for which there exists constants $A$ and $B$ such that 
$$A\|f\|^2 \leq \sum_n{|\langle f,e_n \rangle |^2} \leq B\|f\|^2.$$

Frames are generally regarded as redundant families of vectors: a frame is not necessarily linearly independent, and a frame expansion of a function is not unique. 

\chapter{Randomized domain-adapted wavelets}

  In this chapter we give details about our construction of  two and three dimensional wavelets on an arbitrary domain, and study their properties.  
 This construction  constitutes one of our main contributions to the thesis. 
  
  \section{The structures needed for the construction}
  In order to start defining and shaping filter operators within the lifting framework, we first need the embedded index sets (grids) $\mc K(j)$ for $j\in \mc J$, as defined in Chapter 2. The scale index set will be $\mc J = \mathbb N$, but in practice we will only consider the finite subset $\mc J = \{0,1,2, \cdots N\}$ for some $N \in \mathbb N$.  After constructing  $(\mc K(j))$, we can give the corresponding detail index sets $(\mc M(j))$ and the nested partitionings $\left(S_{j,k}\right)$. We will also need \emph{neighbor, sibling, parent} and \emph{child} relations on $(\mc K(j))$, to be used in the \emph{prediction} and \emph{update} steps of the lifting.

\subsection{The domain}  
\label{subsec:domain}
The domain $\mc X$, which will be an input to the algorithm, is a subset of $\mathbb Z^2$ or $\mathbb Z^3$. For the case of $\mathbb Z^3$, the domain $\mc X$  actually corresponds to the intersection of a set $X\subset \mathbb R^3$ with the discrete grid $\{(n_x d_x, n_y d_y, n_z d_z)\colon (n_x,n_y,n_z)\in\mathbb Z^3\}$, where $d_x$, $d_y$ and $d_z$ are the resolution parameters in the corresponding directions. Each $n\in\mc X$ represents the rectangular region 
\begin{align}
S_n = &[n_x d_x - d_x/2, n_x d_x + d_x/2]\notag\\
      &\times [n_y d_y - d_y/2, n_y d_y + d_y/2]\notag\\
      &\times [n_z d_z - d_z/2, n_z d_z + d_z/2]\label{eqn:defn_sn},
\end{align}
which is called a \emph{voxel}.  A similar definition holds for two dimensions, in which the set $S_n$ is called a \emph{pixel}.   An example of a discrete domain and the corresponding pixels is displayed in Figure \ref{fig:discrete_dom}. 
\begin{figure}
 \centering
 \includegraphics[width=12cm]{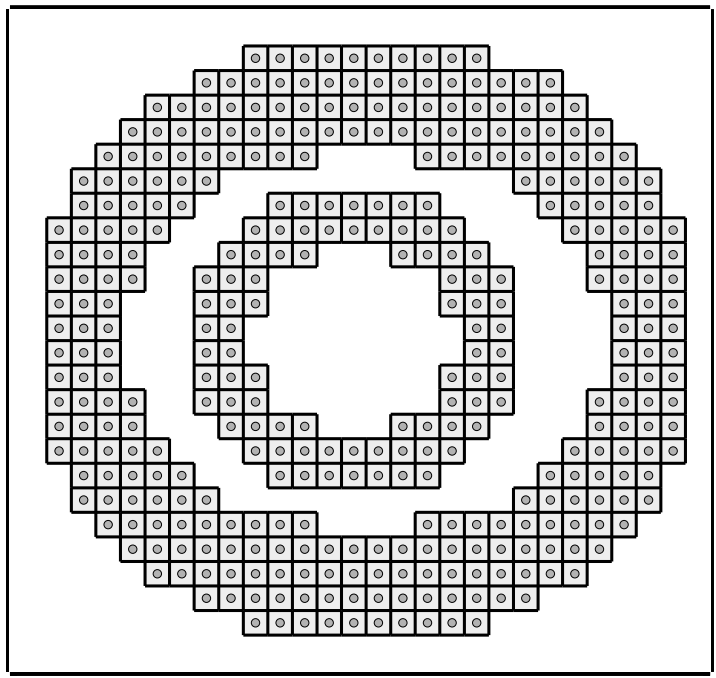}
 \caption{A  two dimensional discrete domain $\mc X$, whose elements are displayed as dots, and the corresponding pixels $(S_n)$ are displayed as the squares surrounding the dots.}
 \label{fig:discrete_dom}
\end{figure}
\subsubsection{The measure}
The measure of each set $S_n$, denoted by $\mu(S_n)$, is another required input. In the regular setting, one takes $\mu(S_n)=1$ for all $n\in\mc{X}$, but in some applications it may be more meaningful to weight each voxel according to some related measure.   

Any set that we shall consider in the finite setting will be a union of the sets $(S_n)$, i.e sets of the form 
\[
 S = \bigcup_{n\in \mc{X'}}S_n
\]
for some $\mc{X'}\subset \mc X$, whose measure is given by
\[
 \mu(S) = \sum_{n\in\mc{X'}}\mu(S_n).
\]

\subsubsection{Neighboring elements}
We define two elements $n, m \in \mc X\subset \mbb Z^3$ to be \emph{neighbors}, if 
\[
 |n_x-m_x|+|n_y-m_y| + |n_y-m_y| = 1,
\]
where $n = (n_x,n_y,n_z)$ and $m = (m_x, m_y, m_z)$.   This is equivalent to the condition that the rectangular prisms bounding regions $S_n$ and $S_m$ share a face.  A similar definition also holds for two dimensions. 

\subsection{Embedded grids and the random merging algorithm}
We choose an $N\in \mathbb N$, which is the number of decomposition levels.  We take the scale index set $\mc J$ to be $\{0,1,2,\cdots,N\}$, where $j=N$ represents the scale of highest resolution available. 

We set $\mc K(N)$ to be equal to the discrete domain  $\mc X$.  For every $k\in\mc K(N)$, we denote the set of {neighbors} of $k$ with \(\operatorname{Nbr}(N,k)\subset \mc K(N) \), based on the neighborhood structure of $\mc X$.  Also we set $S_{N,k} := S_k$, where $S_k$ is the set defined in (\ref{eqn:defn_sn}). After defining $\mc K(N)$, we will define $\mc K(N-1), \mc K(N-2), \cdots ,\mc K(0)$, and the corresponding neighborhood structure with a random merging algorithm, inductively.
  
For a given $j$, assume we have $\mc K(j)$, the sets $(S_{j,k})_{k\in\mc K(j)}$, and a collection of sets $\{\opnm{Nbr}(j,k)\colon k\in\mc K(j)\}$ of the neighbors for each element.
\begin{enumerate}
\item  Start with $\mc K(j-1) = \mc M(j-1) = \emptyset$. 
 \item Compute the centroids of all sets $S_{j,k}$, which is defined as 
 \[
      C(S_{j,k}) = \frac{1}{\mu(S_{j,k})}\sum_{S_{N,k}\subset S_{j,k}}(k_x d_x, k_y d_y, k_z dz)\mu(S_{N,k}),
 \]
 which gives a vector $C(S_{j,k})\in\mathbb R^3$.
\item \label{step:randperm} Declare all elements of $\mc K(j)$ as available, and put them in a linear order that is determined by a random permutation,
 \item \label{step:togo}For  the first available $k\in\mc K(j)$, select at most $p$ of the available neighbors, ${s_1,s_2,\cdots, s_q}\in\mc K(j)$, where $q\leq p$,
  in a random way, with probability of being selected being inversely proportional to the distance between the centroids,
 \item Add $k$ to $\mc K(j-1)$ and each of $s_1, s_2, \cdots, s_q$ to $\mc M(j-1)$,
  \item
  Mark each of ${k,s_1,s_2,\cdots, s_q}\in\mc K(j)$ to be \emph{unavailable}, and declare them to be \emph{siblings}, a relationship denoted as
 \[
  \opnm{Sib}(j,k) = \opnm{Sib}(j,s_1) = \cdots =\opnm{Sib}(j,s_q)= \{k,s_1,s_2,\cdots,s_q\},
 \]

 \item Define the set $\disp S_{j-1,k} := \bigcup_{l\in\opnm{Sib}(j,k)} S_{j,l},$
 \item Go to Step \ref{step:togo}, until all elements of $\mc K(j)$ are marked as unavailable,
 \item Declare two elements $k_1, k_2\in \mc K(j-1)$ to be neighbors, if there exists $s_1 \in \opnm{Sib}(j,k_1)$ and $s_2 \in \opnm{Sib}(j,k_2)$ such that $s_1$ and $s_2$ are neighbors in $\mc K(j)$, i.e., $s_1 \in \opnm{Nbr}(j,s_2)$.

\end{enumerate}
There are two sources of randomness in this algorithm: one is in the random permutation of $\mc K(j)$ in Step \ref{step:randperm} and the other is in the random choice of neighboring sets to be selected to be merged in Step \ref{step:togo}. 

It can be verified that the set of partitionings $(S_{j,k})$ that is output by the algorithm, after being extended for levels $j = N+1,N+2,\ldots$  in some suitable way, satisfies Definition \ref{defn:partitionings}.

The parameter $p$ in Step \ref{step:togo} is mostly set to 3 in our experiments.  

The formation of siblings is displayed in Figure \ref{fig:grid_siblings}, an example output of the algorithm in Figure \ref{fig:embedded_grid}, and the graphs indicating the neighbors is shown in \mbox{Figure \ref{fig:grid_neighbors}}. 
\begin{figure}
 \centering
 \includegraphics[width=8cm]{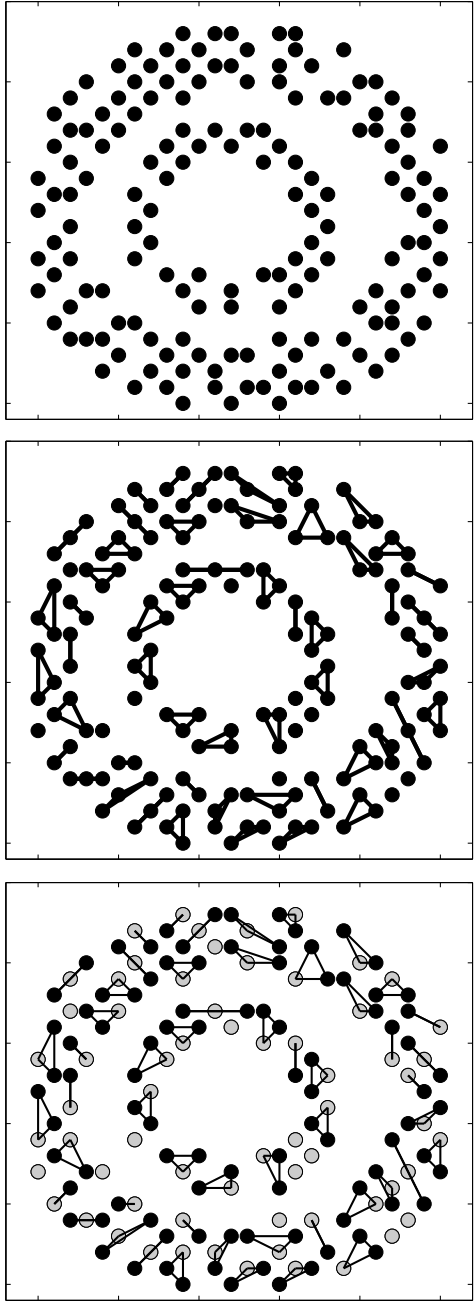}
 \caption{Illustration of merging and splitting of the grid $\mc K(j)$.  The top row is the original grid $\mc K(j)$. In the second row, elements of the grid are randomly merged with  neighbors, and they form groups consisting of one, two or three elements.  In the last row, one element in each group, determined in random way, is displayed in a lighter color.  The elements shown in light color make up $\mc K(j-1)$, and grid elements that remain dark-colored belong to the complimentary grid $\mc M(j-1)$. }
 \label{fig:grid_siblings}
\end{figure}

\begin{figure}
 \centering
 \includegraphics[width=15cm]{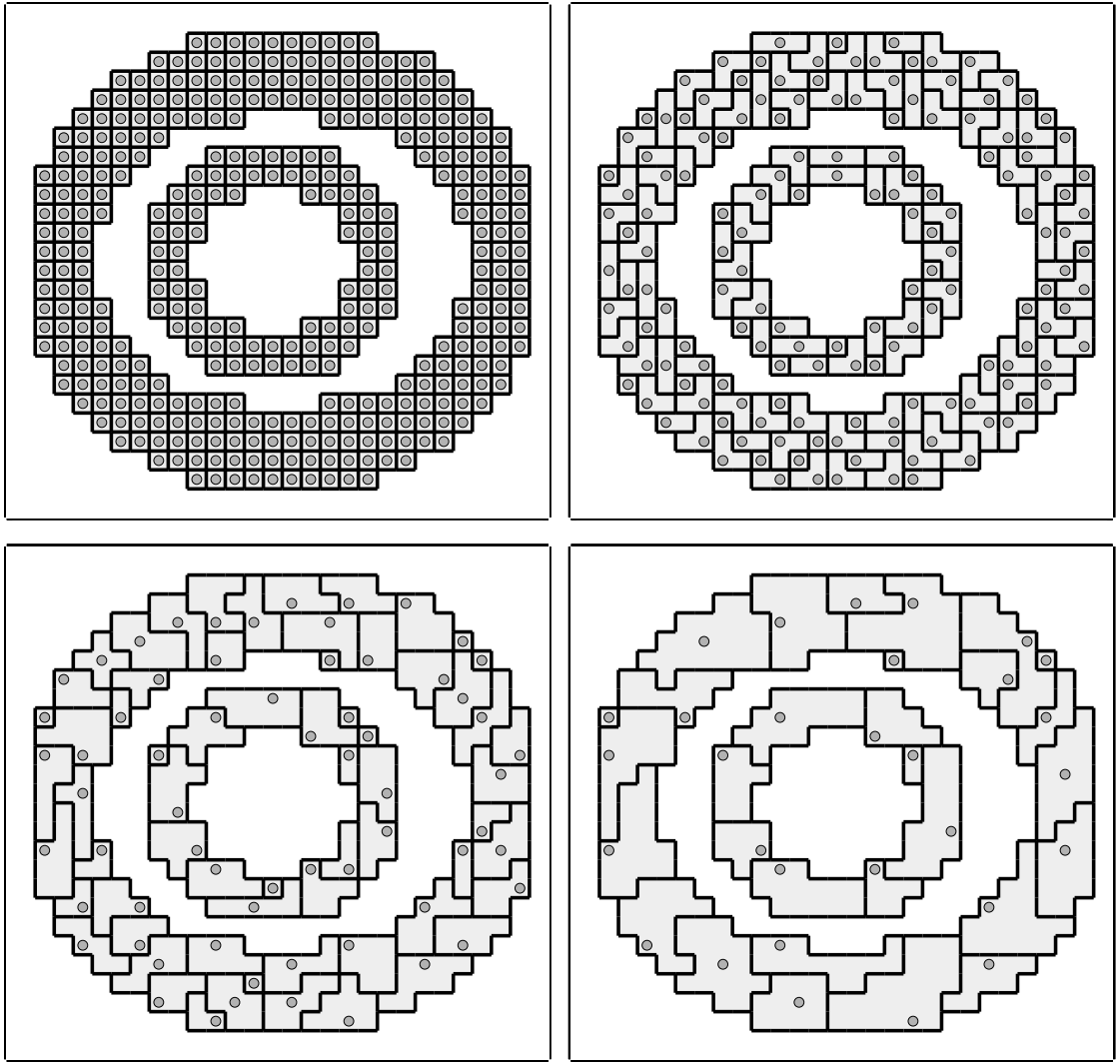}
 \caption{The embedded grid and the corresponding regions for a three level decomposition.  In each figure, the dots represent $\mc K(3), \mc K(2), \mc K(1)$ and $\mc K(0)$ in the reading order. The shaded regions bounded by lines represent 
 $(S_{3,k})_{k\in\mc K(3)}, (S_{2,k})_{k\in\mc K(2)}, (S_{1,k})_{k\in\mc K(1)}$ 
 and $(S_{0,k})_{k\in\mc K(0)}$.
 }
 \label{fig:embedded_grid}
\end{figure}

\begin{figure}
 \centering
 \includegraphics[width=11cm]{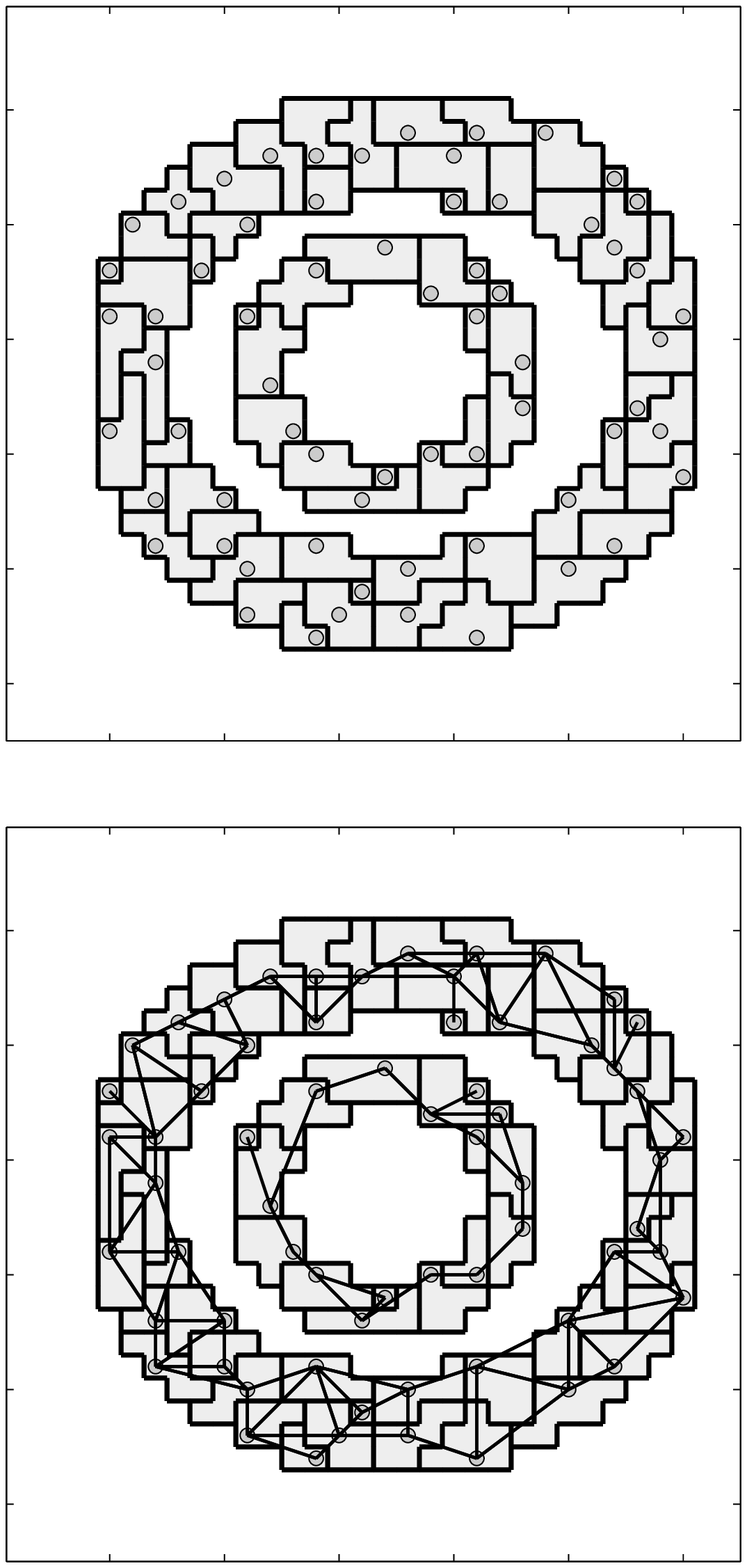}
 \caption{Grid neighbors and regions.  In the top each grid element $k\in\mc K(j)$ is shown insinde the region $S_{j,k}$ it represents.  The bottom row displays the neighborhood graph, where the condition for two elements to be neighbors is equivalent to the regions they represent sharing a nontrivial boundary.  Note that there are grid element pairs that are geographically very close but that nevertheless are not neighbors of each other.}
 \label{fig:grid_neighbors}
\end{figure}


\section{The wavelet transform}
After obtaining the necessary structures, such as the nested index sets $(\mc K(j))_j$, their complements $(\mc M(j))_j$, and the neighborhood and sibling relations on them, we are now ready to define the filter operators that identify the wavelet transform.  Our wavelet transform  is basically a version of David Donoho's \emph{average interpolating wavelets} \cite{donoho}, which can be obtained with a three-stage lifting framework \cite{Sweldens_buildingyour}. Following  the splitting step (lazy wavelet transform), it continues with a prediction, an update and another prediction step.  The first two steps are  similar to Haar decomposition, and result in wavelets similar to the ones called \emph{unbalanced Haar wavelets} \cite{unbalanced_haar}. These wavelets are orthogonal, and the transform can be considered to be a meaningful wavelet transform in its own right.  

The second prediction step is a version of the average interpolation prediction, which converts the Haar-like wavelets of the previous step into smoother functions.  With this step, we lose orthogonality, in return for smooth wavelet and scaling functions, which entails better sparsity properties for the output. 

\subsection{The finest available resolution}
In the practical implementation, the  scale index $\mc J = \{1,2,\cdots,N\}$ is finite, there is a highest available resolution, which is determined by the measuring device's capabilities. 
The discrete function input to the transform $f_d\in \ell^2(\mc X)$ is assumed to come from a corresponding function $f\in L^2(X)$, with
\beqn
\label{eqn:finest_dual_scaling}
 f_d(n) = \frac{1}{\mu(S_{n})} \int_{S_n}f\, d\mu ,\ \  \ \ \text{for } n\in \mc X
\eeqn
for some measure that agrees with the  $(\mu(S_n))_n$ of Subsection \ref{subsec:domain}.   As such, the dual scaling functions of the finest level $N$ are defined to be
\[
 \tilde \varphi_{N,k} = \frac{1}{\mu(S_{N,k})}\chi(S_{N,k}),
\]
and the primal scaling functions $(\varphi_{N,k})_{k\in\mc K(N)}$ are
\[
  \disp\varphi_{N,k} = \chi (S_{N,k}),
\]
where $\chi$ denotes the characteristic function, which takes $1$ inside the set and $0$ elsewhere.   Accordingly, the finest resolution scaling function coefficients $\lambda_{N,j}$ are simply taken to be equal to the $f_d(n)$'s, and as a result
\[
 \sum_{k\in\mc K(N)}\lambda_{N,k}\varphi_{N,k} = \sum_{n\in\mc X}f_d(n)\chi(S_n) \approx f,
\]
noting that $\mc K(N) = \mc X$ and $S_k = S_{N,k}$.

In synthesizing the scaling functions and the wavelets with the algorithm in Subsection \ref{subsec:scaling_fun_cascade}, we will have to stop at $j = N$.  Hence, in a sense, the scaling functions $(\varphi_{N,k})_k$ will be the atomic building blocks for the wavelets and scaling functions that we consider, in the same sense that the sets $(S_{N,k})_k$ are the building blocks for all sets that we consider, in the construction.

\subsection{The lifting steps}
We assume an input $\lambda_{j+1} = (\lambda_{j+1,k})_{k\in\mc K(j)}$ is given. The analysis filters will initially perform the lazy wavelet
transform, which is a simple splitting.  This initial pair of analysis filters will be denoted as $(\tilde H^{\text{lazy}}_j, \tilde
G^{\text{lazy}}_j)$ and their outputs as $(\lambda^{(0)}_j,\gamma^{(0)}_j)$.  Three lifting steps (a prediction, an update and another prediction)
will modify the filters. The resulting analysis filter pairs will be denoted as $\left(\tilde H^{(1)}_j, \tilde G^{(1)}_j\right)$, $\left(\tilde
H^{(2)}_j, \tilde G^{(2)}_j\right)$ and 
$\left(\tilde H^{(3)}_j, \tilde
G^{(3)}_j\right)$.  The outputs of these filters will be denoted as  $\left(\lambda^{(1)}_j,\gamma^{(1)}_j\right)$, $\left(\lambda^{(2)}_j
,\gamma^{(2)}_j\right)$ and $\left(\lambda^{(3)}_j,\gamma^{(3)}_j\right)$, respectively.   The lifting operators will be denoted by $P_1$, $U$ and $P_2$, respectively, and the filters will satisfy
{\allowdisplaybreaks
\begin{align*}
 \tilde H^{(1)}_j &= \tilde H^{\text{lazy}}_j
 \\
 \tilde G^{(1)}_j &= \tilde G^{\text{lazy}}_j - P_1 H^{\text{lazy}}_j 
 \\
 \tilde H^{(2)}_j &= \tilde H^{(1)}_j + U \tilde G^{(1)}
 \\
 \tilde G^{(2)}_j &= \tilde G^{(1)}_j 
\\
 \intertext{and,}
  \tilde H^{(3)}_j &= \tilde H^{(2)}_j 
 \\
 \tilde G^{(3)}_j &= \tilde G^{(2)}_j - P_2 \tilde H^{(2)}_j.
\end{align*}
}In the implementation, one does not need to explicitly compute these filters. Starting with $\lambda_{j+1}$ and, one can obtain $\lambda_{j}^{(3)}$ and $\gamma_{j}^{(3)}$ in a sequence of operations
\begin{align*}
 & \lambda_{j}^{(0)} = \tilde H^{\text{lazy}}_j \lambda_{j+1},
 &
 & \gamma_{j}^{(0)} = \tilde G^{\text{lazy}}_j \lambda_{j+1},
 \\
& \lambda_{j}^{(1)} = \lambda_{j}^{(0)},
&
 & \gamma_{j}^{(1)} = \gamma_{j}^{(0)} - P_1 \lambda_{j}^{(0)},
\\
& \lambda_{j}^{(2)} = \lambda_{j}^{(1)} + U \gamma_{j}^{(1)},
&
 & \gamma_{j}^{(2)}  = \gamma_{j}^{(1)},
\\
& \lambda_{j}^{(3)} = \lambda_{j}^{(2)},
&
 & \gamma_{j}^{(3)}  = \gamma_{j}^{(2)} - P_2 \lambda_{j}^{(2)},
\end{align*}
as displayed in Figure \ref{fig:lifting_diagram}. Similarly, the inverse transform is implemented by reversing the steps:
\begin{gather*}
\begin{aligned}
  & \lambda_{j}^{(2)} = \lambda_{j}^{(3)}, \hspace{3cm}
  &
  & \gamma_{j}^{(2)}  = \gamma_{j}^{(3)} + P_2 \lambda_{j}^{(2)},
  \\
  &\gamma_{j}^{(1)}  = \gamma_{j}^{(2)},
  &
  & \lambda_{j}^{(1)} = \lambda_{j}^{(2)} - U \gamma_{j}^{(1)},
  \\
  & \lambda_{j}^{(0)} = \lambda_{j}^{(1)},
&
 & \gamma_{j}^{(0)} = \gamma_{j}^{(1)} + P_1 \lambda_{j}^{(0)},
 \end{aligned}
 \\
  \lambda_{j+1} =  H^{*\, \text{lazy}}_j\lambda_{j}^{(0)} +
      G^{*\, \text{lazy}}_j\gamma_{j}^{(0)}.
\end{gather*}
The filters $\left(\tilde
H^{(2)}_j, \tilde G^{(2)}_j\right):=\left(\tilde H^{\text{Haar}}_j, \tilde G^{\text{Haar}}_j\right)$ that come after the second lifting step 
actually corresponds to the unbalanced Haar transform, which we will give in the next subsection.   The third lifting step is based on \emph{average interpolation},  and the resulting filters after this step, $\left(\tilde
H^{(3)}_j, \tilde G^{(3)}_j\right):=\left(\tilde H^{\text{AI}}_j, \tilde G^{\text{AI}}_j\right),$  are called  average interpolating filters, to be explained in Subsection \ref{subsec:average_interp}. 

\begin{figure}
 \centering
 \begin{overpic}[scale=.53]
 {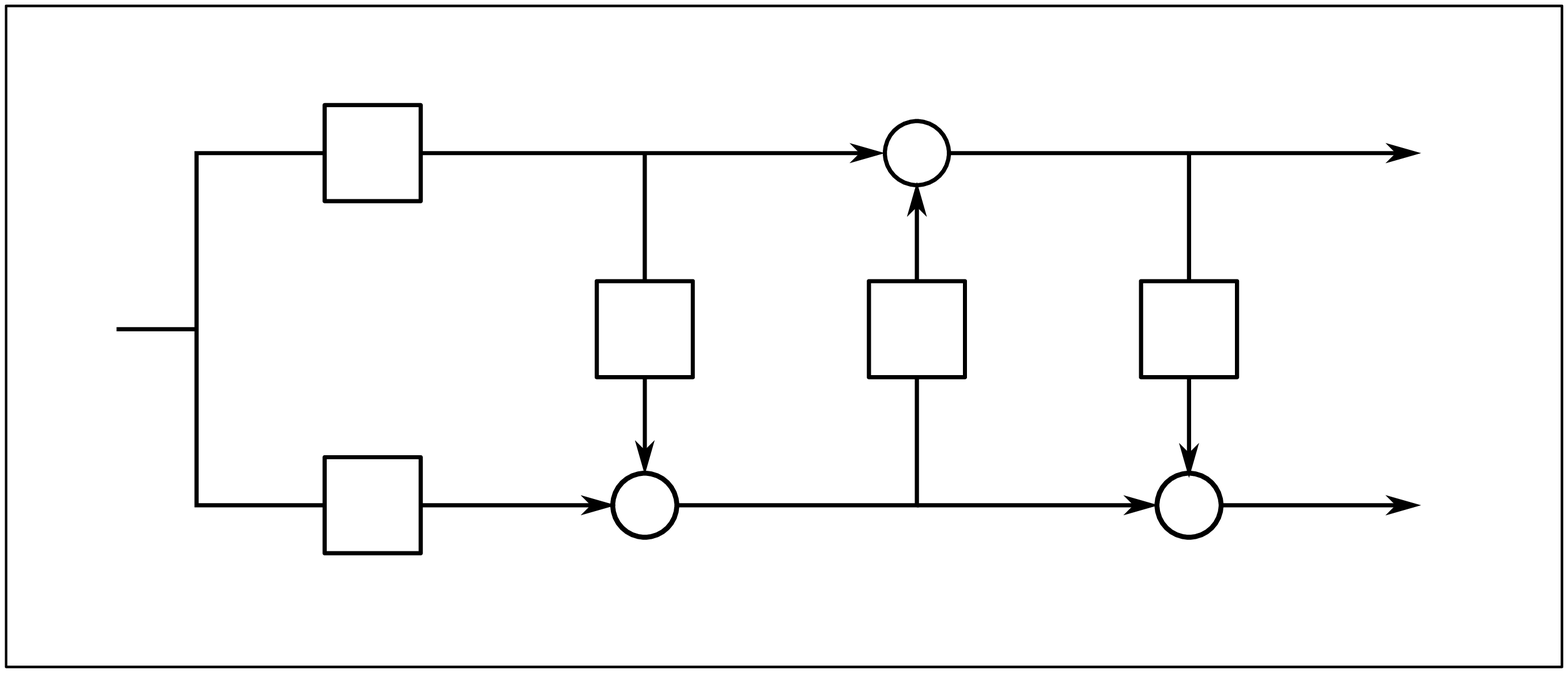}
 \put(1.5,21){$\lambda_{j+1}$}
  \put(39.5,20.5){$P_1$}
  \put(57.5,20.5){$U$}
  \put(74.5,20.5){$P_2$}
  \put(38,13){$-$}
  \put(36.5,11.5){$+$}
  \put(73,13){$-$}
  \put(71.5,11.5){$+$}
       
  \put(54.5,30.5){$+$}
  \put(56,29){$+$}
  
 \put(21,32){\footnotesize{$\tilde H_j^{\text{lazy}}$}}
 \put(92,32.25){$\lambda^{(3)}_{j}$}
 \put(32,34.75){$\lambda^{(0)}_{j}$}
  \put(48,34.75){$\lambda^{(1)}_{j}$}
    \put(66,34.75){$\lambda^{(2)}_{j}$}

\put(32,6.5){$\gamma^{(0)}_{j}$}
  \put(48,6.5){$\gamma^{(1)}_{j}$}
    \put(66,6.5){$\gamma^{(2)}_{j}$}

 \put(57.75,32.25){\footnotesize{$\Sigma$}}
  \put(40.25,9.5){\footnotesize{$\Sigma$}}
 \put(75,9.5){\footnotesize{$\Sigma$}}

\put(21,9){\footnotesize{$\tilde G_j^{\text{lazy}}$}}
\put(92,9.75){$\gamma^{(3)}_{j}$}
 \end{overpic}
 \caption{Diagram illustrating the implementation of the three-stage lifting.}
 \label{fig:lifting_diagram}
\end{figure}

\subsection{Unbalanced Haar wavelets}
\subsubsection{First prediction}
The first prediction operator $P_1$, is a function from $\ell^2(\mc K(j))$ to $\ell^2(\mc M(j))$.   First we note that, for each $m\in \mc M(j)$ there exists a unique $k_m\in\mc K(j)$ such that $k_m \in \opnm{Sib}(j+1,m)$.  If $a = (a_k)_{k\in\mc K(j)}$ is the input to $P_1$, the  $m$th entry in the output of $P_1$ is simply defined to be equal to $a_{k_m}$, i.e.,
\begin{align*}
   P_1\colon & \ell^2( \mc K(j)) \rightarrow  \ell^2( \mc M(j)), \ \ \  P_1 a = b \ \text{means}
\\
&\hspace{2.5cm} b_{m} = a_{k_m}, 
\end{align*}
where $k_m$ is the unique element of $\mc K(j) \cap \opnm{Sib}(j+1,m)$. 

The first prediction step is based on predicting the value of an entry  $\lambda_{j+1, k}$ to be the same value as $\lambda_{j+1,k_m}$, for a specific neighbor $k_m$ in $\mc K(j+1)$. The idea is very similar to \emph{nearest neighbor interpolation}, except that the neighbor used in prediction is not necessarily the nearest one, but is determined by the structure of the embedded grids.  For the example in Figure \ref{fig:grid_siblings},  the value at the dark colored spots in the lowest row is to be predicted with the value in the corresponding light colored spot. 

 We note that, after this first prediction step, one has
\begin{align}
\label{eqn:after_first_pred}
 \gamma_{j,s_1}^{(1)} = &\lambda_{j+1,s_1} - \lambda_{j+1,k}\notag\\
& \vdots \\
  \gamma_{j,s_q}^{(1)} = &\lambda_{j+1,s_q} - \lambda_{j+1,k}.\notag
\end{align}
where $k\in\mc K(j)$ and  $\opnm{Sib}(k,j+1) = \{k, s_1, s_2, \cdots, s_q\}$, with $q$ being the number of siblings.

\subsubsection{Update}
The update operator $U$ is a function from $\ell^2(\mc M(j))$ to $\ell^2(\mc K(j))$ that modifies the output of the previous step as $\lambda^{(2)}_j = \lambda^{(1)}_j + U \gamma^{(1)}_j$, in order to guarantee that
\[
\sum_{k\in\mc K(j+1)} \lambda_{j+1,k}\,\mu(S_{j+1,k})=\sum_{k\in\mc K(j)} \lambda^{(2)}_{j,k}\,\mu(S_{j,k}).
\]
It is sufficient, and usually expected, that this condition is satisfied locally, by having 
\begin{align}
\label{eqn:measure_preserve_loc}
 \lambda_{j,k}^{(2)}\,\mu(S_{j,k}) 
 =&
 \lambda_{j+1,k}\,\mu({S_{j+1,k}})
  + \lambda_{j+1,s_1}\,\mu({S_{j+1,s_1}}) 
     \notag\\& + \cdots + \lambda_{j+1,s_q}\,\mu({S_{j+1,s_q}}),
\end{align}
for all $k\in\mc K(j)$, and $s_1,s_2,\cdots,s_q\in  \opnm{Sib}(k,j+1)$.
This can be accomplished with an update operator $U$ given by
\begin{align*}
  U_j\colon & \ell^2( \mc M(j)) \rightarrow  \ell^2( \mc K(j)), \ \ \  U a = b \ \text{means}
\\
& b_{k} = 
\frac
{a_{s_1}\,\mu(S_{j+1,s_1})+\cdots+a_{s_q}\,\mu(S_{j+1,s_q})}
{\mu(S_{j+1,k})+\mu(S_{j+1,s_1})+\cdots+\mu(S_{j+1,s_q})},
\end{align*}
where $\{k,s_1,s_2,\cdots,s_q\} = \opnm{Sib}(j+1,k)$ for all $k \in \mc K(j)$. It can be verified that after this step, one has
\begin{equation}
\label{eqn:averaging_property}
 \lambda_{j,k}^{(2)} = 
 \frac{\lambda_{j+1,k}^{(0)}\,\mu(S_{j+1,k})+\lambda_{j+1,s_1}^{(0)}\,\mu(S_{j+1,s_1})+\cdots+\lambda_{j+1,s_q}^{(0)}\,\mu(S_{j+1,s_q})}{\mu(S_{j+1,k})+\mu(S_{j+1,s_1})+\cdots+\mu(S_{j+1,s_q})},
\end{equation}
which is the same as (\ref{eqn:measure_preserve_loc}), noting that 
\begin{equation*}
\mu(S_{j,k})= \mu(S_{j+1,k})+\mu(S_{j+1,s_1})+\cdots+\mu(S_{j+1,s_q}).
\end{equation*}

\subsubsection{Wavelets and scaling functions}
Using \eqref{eqn:finest_dual_scaling} and \eqref{eqn:averaging_property}, it can be shown inductively that the dual scaling functions $\tilde \varphi_{j,k}$ are given by
\[
 \tilde \varphi_{j,k} = \frac{1}{\mu(S_{j,k})} \chi(S_{j,k}),
\]
and using \eqref{eqn:after_first_pred} one gets the dual wavelets  $(\tilde \psi_{j,m})$ as
\[
 \tilde \psi_{j+1,\,m} = \frac{1}{\mu(S_{j+1,\,k_m})}\chi(S_{j+1,\,k_m}) - \frac{1}{\mu(S_{j+1,m})}\chi(S_{j+1,m}),
\]
where $k_m$ is the unique element of $\opnm{Sib}(j+1,m)\cap \mc K(j)$. 

It is easy to verify that the collection of dual scaling functions and wavelets are orthogonal to each other, which implies that the primal scaling functions and wavelets are the same as the dual ones, up to normalizing constants. These type of wavelets are studied in \cite{unbalanced_haar}, where they are named \emph{unbalanced Haar wavelets}.

\subsection{Average interpolating wavelets}
\label{subsec:average_interp}  The unbalanced Haar wavelets have the advantage of being orthogonal, and they are relatively simple to design, but they have the drawback of not being
  smooth, not even continuous. In a wavelet representation
  $$
  f =
  \sum_{j,m}\gamma_{j,m}\psi_{j,m},
  $$
  we
  usually seek \emph{sparse representations}, meaning that most of the coefficients
  $\gamma_{j,m}$'s are zero or at least negligible. This is usually not achieved by
Haar
  wavelets to the most pronounced degree, because there is a link between the
  smoothness of wavelets and fast decay of wavelet coefficients. Intuitively, we
  can say that Haar wavelets, not being smooth, find it difficult to
represent
  smooth functions and require a larger number of basis elements than is needed by a more smooth basis.  Therefore we
would
  like to have a wavelet basis without sharp discontinuities and certain
  smoothness properties.    
  
  By adding a second prediction operator $P_2$ to the lifting scheme, it is possible to obtain
smoother wavelets out of Haar wavelets, and that give smaller detail coefficients for smooth signals. 
%
  Or, when we think the other way around, if the prediction operator $P_2$  is designed to make the entries of the detail output   smaller for smooth functions, then the corresponding wavelets will turn out to be smoother than Haar wavelets. 
  
  The first prediction step $P_1$ was based on a version of \emph{nearest neighbor interpolation}, while this second prediction step is based on the notion of \emph{average interpolation}, which we explain next.  The use of average interpolation within the context of wavelets is due to Donoho \cite{donoho}.

\subsubsection{Average interpolation}
In the following, all sets considered are subsets of $\mathbb R^n$, and $\mu$ is a measure on $\mathbb R^n$.  The average $ \frac{1}{\mu(S)}\int_S f d\mu$ is denoted shortly as  $ \dashint_S f$, for a given $S\subset \mathbb R^n$. 

Let $S_1,S_2,\cdots, S_r,{C_1},C_2,\cdots, {C_t}\subset \mathbb R^n$ be mutually disjoint sets that are geographically close, for some integers $r$ and $t$.  
Let $f$ be an unknown function, which is assumed to be locally well approximated by polynomials.  

It is assumed that the averages $\dashint_{S_1} f, \cdots, \dashint_{S_r} f$ are known, and the goal is to estimate $\dashint_{C_1} f, \cdots, \dashint_{C_t} f$.  

The approach is to find a polynomial $p$ whose average values agree with $f$ on on $S_1,\cdots,S_r$, i.e.,
\begin{equation}
\label{}
\dashint_{S_i} f = \dashint_{S_i} p 
\end{equation}
for  $i = 1,2,\cdots,r.$  After such a polynomial is found, the estimates for $\dashint_{C_1} f, \cdots, \dashint_{C_t} f$ are given to be $\dashint_{C_1} p, \cdots, \dashint_{C_t} p$, respectively.

\vspace{.5cm}
After this brief description of average interpolation, we can define the second prediction operator that is used to obtain the average interpolating wavelets. 
\subsubsection{Second prediction}
The second prediction operator $P_2$ maps $\ell^2(\mc K(j))$ to $\ell^2(\mc M(j))$, aiming to estimate $\gamma^{(2)}_j$ using $\lambda^{(2)}_j$. Here using \eqref{eqn:averaging_property} and assuming the finest level coefficients $\lambda_{N,j}$ are local averages of some input function $f$ as in  \eqref{eqn:finest_dual_scaling}, we get
\[
 \lambda_{j,k}^{(2)} = \dashint_{S_{j,k}} f.
\]
Also for $m\in\mc M(j)$, by \eqref{eqn:after_first_pred}, one has
\begin{align*}
   \gamma_{j,m}^{(2)} = \gamma_{j,m}^{(1)}  &= \lambda_{j+1,m} - \lambda_{j+1,k_m}\\
	&= \dashint_{S_{j+1,m}}\!\!\!\!f\ \ \  \ -\  \ \ \   \dashint_{S_{j+1,k_m}}\!\!\!\!f,
\end{align*}
where $k_m$ is the unique element of $\mc K(j)\cap\opnm{Sib}(j+1,m)$. We aim to estimate this value, using $\lambda^{(2)}_{j,n_1},\cdots, \lambda^{(2)}_{j,n_r}$, where   $\{n_1,n_2,\cdots,n_r\}=\opnm{Nbr}(j,k_m)$.  This is actually the problem of   estimating $\dashint_{S_{j+1,m}} f$ and 
$\dashint_{S_{j+1,k_m}}f$ using the values $\dashint_{S_{j,n_1}}f, \cdots, \dashint_{S_{j,n_r}}f$, which is to be handled by average interpolation.

In our implementation, for each $k\in\mc K(j)$ we solve for the first degree polynomial
\begin{equation}
\label{eqn:polynomial}
 p_{j,k}(x,y,z) = a_{j,k} + b_{j,k}\, x + c_{j,k}\, y + d_{j,k}\, z
\end{equation}
that minimizes 
\[
 \sum_{n\in\opnm{Nbr}(j,k)}\left|  \lambda^{(2)}_{j,n} - \dashint_{S_{j,n}} p_{j,k}\right|^2. 
\]
 This is a linear problem, and each of the polynomial coefficients in \eqref{eqn:polynomial} is a linear combination of $\lambda^{(2)}_{j,n_1},\cdots, \lambda^{(2)}_{j,n_r}$.  So we can define $P_2$ to be the linear operator such that the $m$th entry of $P_2 \lambda_j^{(2)}$ is given by
\[
 \dashint_{S_{j+1,m}}\!\!\!\!p_{j,k_m}\ \ \  \ -\  \ \ \   \dashint_{S_{j+1,k_m}}\!\!\!\!p_{j,k_m},
\]
noting that this quantity is a also a linear combination of $\lambda^{(2)}_{j,n_1},\cdots, \lambda^{(2)}_{j,n_r}$. If the input $f$ to the algorithm is itself a first degree polynomial, then the fit would be perfect and the prediction will give exact quantities.

After this final prediction step, one can add a normalizing step to the transform that makes each of the wavelets and the dual wavelets of unit norm.

\section{Numerical experiments}

 \subsection{The potential to reduce noise by sparse representation}
  The underlying point in transform-based noise reduction methods is the sparse representation property. A good transform enables a representation that puts  most of the information about a smooth function into a relatively small fraction of the coefficients, while  spreading out a signal uniformly to all components if the signal is noise-like, having little spatial correlation.  More explicitly if we denote the wavelet representation of a function with a simplified single indexed notation as
  \[
   f = \sum_{m\in\mc M}\gamma_m \psi_m,
  \]
then we assume that for the $f$ of interest to us, there exists a set $\mc A \subset \mc M$ with $|\mc A|\ll |\mc M|$ that gives
 \begin{equation}
 \label{eqn:thresholded_approximation}
   f_{\text{approx}}:= \sum_{m\in\mc A}\gamma_m \psi_m \approx f .
 \end{equation}
  For the case of orthogonal wavelets, the best choice for such sets are generally of the form
 \begin{equation}
 \label{eqn:thresholded_set}
   \mc A = \{m\in\mc M \colon |\gamma_m|>\tau\}, 
 \end{equation}
for some threshold $\tau$.  These type of sets give the best approximation in the $L^2$ sense, among the sets of the same size. This way of choosing $\mc A$ usually still works even when the wavelet basis is not orthogonal for some signal subclasses, if the corresponding frame bounds are close to tight. 
\subsubsection{Comparing adaptive wavelets with standard wavelets}
In this numerical experiment, we started with a domain $X$  consisting of concentric rings, contained in a rectangle $R$, as displayed in the top row of Figure \ref{fig:wav_sparsity_comp}.  We generated a smooth function on $R$, which is a different random linear combination of two dimensional Gaussians on each connected component of the $X$.  On $R\setminus X$ the domain also contains another smooth function of the same type.  That is to say, if the domain can be written as $X = S_1\cup\cdots\cup S_n$, where each of $S_1,\cdots,S_n$ is a single, connected ring, and if $f_1,\cdots,f_{n+1}$ are smooth functions on $R$, then our input $f$ is taken to be 
\[
\label{eqn:random_smooth}
 f = f_1\,\chi_{S_1} + f_n\,\chi_{S_n} + f_{n+1}\,\chi_{R\setminus X},
\]
$\chi_S$ denoting the characteristic function of a set $S$. 

This choice of $f$ is motivated by the fMRI problem, where the brain cortex, which is the natural domain of the measured data, has a convoluted structure, and geographically close parts of it may carry signals that are of different nature. 

We computed the coefficients under different wavelet transforms, and computed reconstructed approximations in the form of \eqref{eqn:thresholded_approximation}, the coefficient sets $\mc A$ being selected with thresholding as in \eqref{eqn:thresholded_set}.  For each different threshold, we computed the approximation error as
\[
  \frac {\|f-f_{\text{approx}}\|}{\|f\|}.
\]
Although the standard wavelet transforms give reconstructions over the whole rectangle $R$, we compute the norms  only over the domain $X$, i.e.,
\[
 \|f\| = \sqrt{\int_X |f(x)|^2 dx}.
\]
We also generated multiple realizations of the noise $n$ over the whole square,
whose samples are i.i.d. standard Gaussian random variables.  We transformed
each of the realizations of the noise with the same wavelets, and reconstructed it using only entries from
the coefficient set $\mc A$ that was determined by the function $f$, and
computed the expected norm $\text{E}[{\|n_{\text{rec}}\|}]$, of the
reconstructed noise.  That is, we process the signal $f$ and the noise $n$
separately, and plot $\frac {\|f-f_{\text{approx}}\|}{\|f\|}$ versus
$\text{E}[{\|n_{\text{rec}}\|}]$, while the threshold parameter $\tau$ is being
reduced. This gives a measure of the {potential to reduce noise} of the related
transform on the domain $X$.  

The results are plotted in Figure \ref{fig:wav_sparsity_comp}, where the
experiment is done with three standard wavelet bases that have the domain $R$,
and the domain-adapted wavelets that are constructed for $X$. For each choice of
wavelets, the experiment is repeated for the decomposition levels 1,2 and 3. 
The results show that increasing the decomposition level gives a much more
pronounced positive effect with the domain adapted wavelets than the standard
wavelets.

\subsubsection{Comparing different domains}
We repeated the experiment while changing the thickness of each of the rings of the domain, or the gap between the rings corresponding to the off-domain regions. The results are given in Figure \ref{fig:four_domains}.  This experiment shows that, in order for the adapted wavelets to have a clear advantage over the standard ones, the rings or the gaps between them must be sufficiently thin. For an example like the one given in the first row, the adapted wavelets do not have a considerable advantage over the wavelets that have the whole rectangle $R$ as their domain.

  \begin{figure}
 \centering
 \includegraphics[width=13cm]{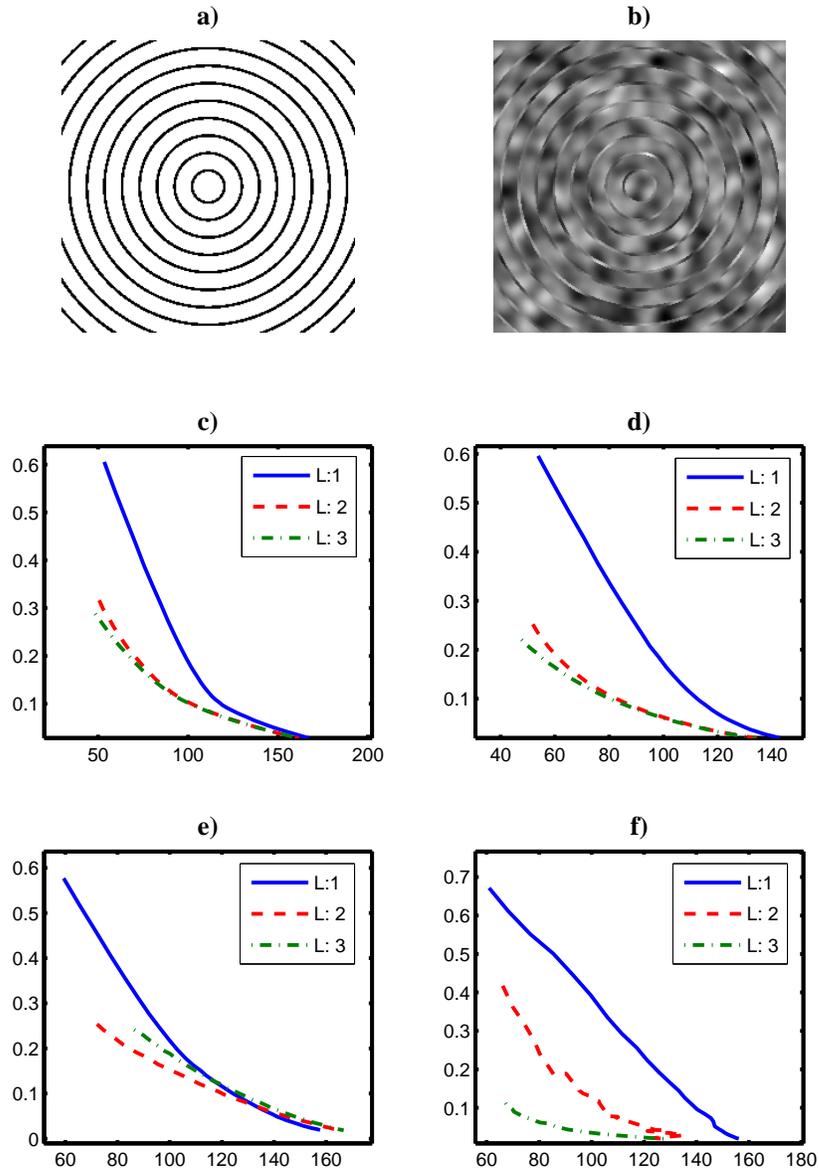}
 \caption{Comparing denoising potentials of different conventional wavelets, and the domain-adapted wavelets.  Relative approximation error (vertical) versus expected noise norm (horizontal), as the threshold is reduced. All error norms are obtained from the surviving coefficients of the original image, and are computed only on the annular domain.  A curve in the lower part of the plane implies a better performance.  a) Domain, b) Image, c) Haar wavelets d) Daubechies-3 wavelets e) Biorthogonal 3.3 wavelets f) Domain-adapted wavelets. The vertical axis in figures c-g is the relative $L_2$ approximation error, which is 
 $\frac {\|f-f_{\text{approx}}\|}{\|f\|}$. Note that increasing the level of the transform has a stronger effect with the domain-adaptive wavelets. 
 }
 \label{fig:wav_sparsity_comp}
\end{figure}

\begin{figure}
 \centering
 \includegraphics[width=15cm]{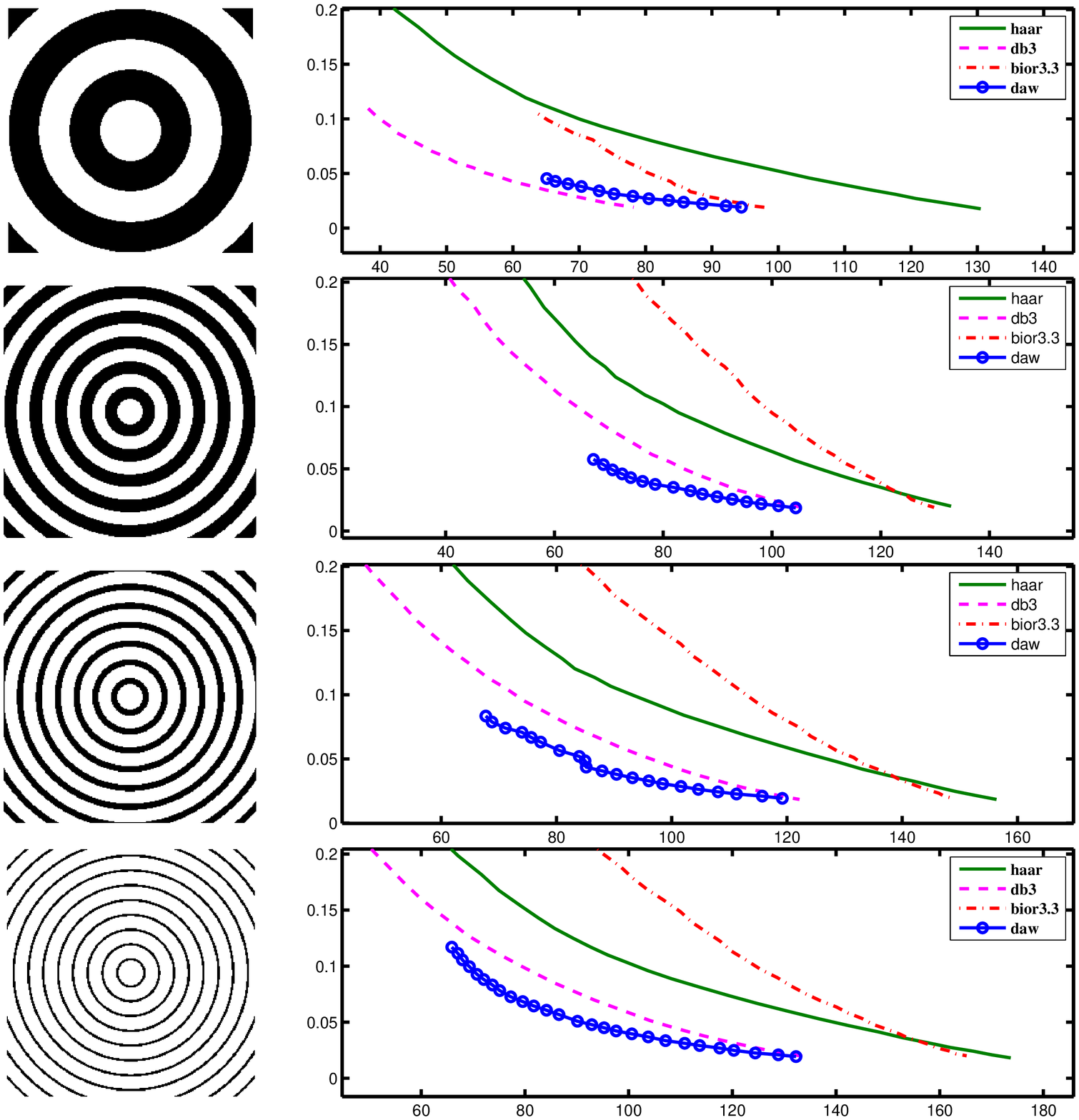}
 \caption{Comparison of the denoising potential of different wavelets on four different two-dimensional domains. Each of the wavelet transforms is computed at Level-3. Figures show the relative approximation error (horizontal) versus expected noise norm (vertical), as the threshold is reduced.  As the domain gets thinner, or the gaps between different circles gets smaller, the domain-adaptive wavelets outperform the conventional wavelets. }
 \label{fig:four_domains}
\end{figure}

\subsection{Improvement due to averaging}
The randomness in our wavelet algorithm allows us to repeat a signal processing task with multiple realization of the wavelets, and then take the average of the results.  In this experiment, we work with a domain consisting of concentric circles and generate smooth signals on them, similar to the examples in the previous subsection.  We add i.i.d. Gaussian noise, and perform a wavelet denoising.  The  signal-to-noise-ratio (SNR) versus number of realizations plot is given in Figure \ref{fig:denoising_with_averaging}.  For chosen domain, with the given noise, the tensor product Daubechies-3 wavelet transform initially performs better denoising than the domain adapted wavelets.  However, as we average over multiple realizations, the domain-adapted wavelets result in better SNR values. 
\begin{figure}
 \centering
 \includegraphics[width=14cm]{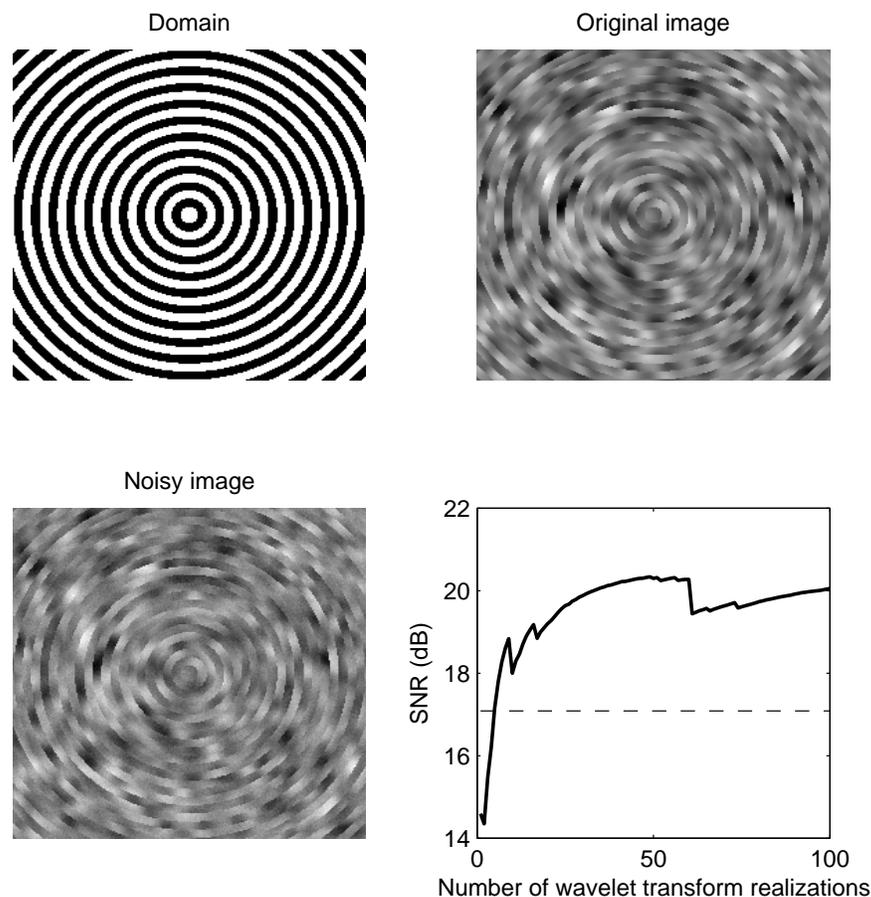}
 \caption{Denoising with multiple wavelet realizations.  The performance improves as we average  the results over multiple realizations.  The dashed line in the plot represent the performance of the Daubechies-3 wavelet, which performed the best among the standard wavelets we tested in the previous subsection. SNR values computed only over the domain. }
 \label{fig:denoising_with_averaging}
\end{figure}

\subsection{The effect of the second prediction}
The second prediction step of the transform is designed to make the detail coefficients smaller, or the wavelets and scaling functions smoother, as explained in Subsection \ref{subsec:average_interp}.  The  wavelet transform before this step is the \emph{unbalanced Haar wavelet} transform, which is orthogonal.  After the second prediction, the wavelets are \emph{average interpolating wavelets}, and they are smooth but not orthogonal.

We tested the effect of this step in noise reduction with an experiment similar to the one in the previous subsection.  The SNR versus number of wavelet realizations is given in Figure \ref{fig:snr_haar_ai}.  This result shows an improvement of nearly 2.5 dB as a result of the second prediction step.  A fixed threshold is empirically determined, and used in all wavelet realizations for the same type.  Since the Haar wavelets are orthogonal, the thresholded approximation gives always the best approximation that can be achieved with the same number of surviving coefficients.  However this is not necessarily the case for the average interpolating wavelets, because they are not orthogonal.  This explains the sudden drops in the SNR curve corresponding to some outlying wavelet realizations.  

\begin{figure}
 \centering
 \includegraphics[width=12cm]{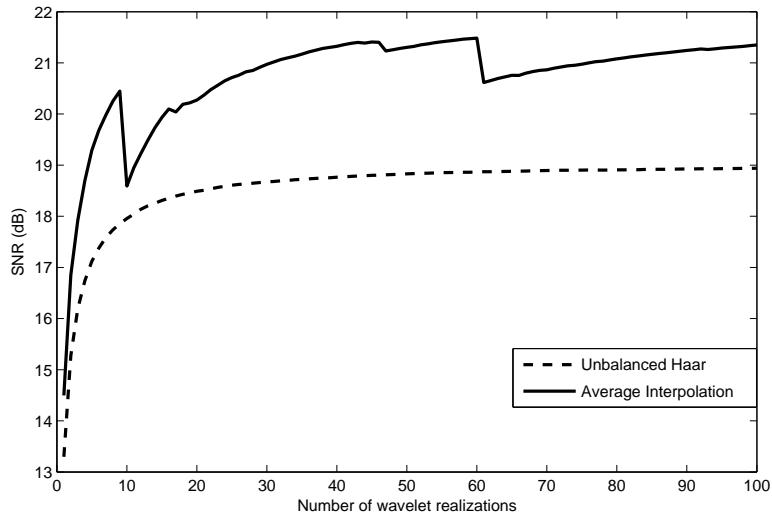}
 \caption{Comparison of unbalanced Haar wavelets and average interpolating wavelets in a noise reduction experiment. Performance improves as the results obtained by newer wavelet realizations are included in the cumulative average. The average interpolating wavelets give better performance than the unbalanced Haar wavelets, but occasional drops in performance are observed when an outlying wavelet realization is encountered. }
 \label{fig:snr_haar_ai}
\end{figure}

\subsection{Translation-invariant processing}
  In a reliable signal processing algorithm, one would intuitively expect a
  translation-invariance property.  If we translate the signal without
  distorting, process it and then translate it back, we would expect the
  result to be the same as if it were processed without being translated. The simplest version of a multiresolution-based wavelet algorithm does not have this invariance, which is why  Donoho and Coifman
  propose  the wavelet spin cycle algorithm in \cite{donoho_transinv}, which achieves translation invariant
  denoising in one dimension.  Translational invariance in higher dimensions can be achieved in the same way.  In dimensions two and higher, one can similarly desire rotational invariance, which can be achieved by introducing redundancy in angular resolution as well, e.g., via steerable filters \cite{freeman91} or dual-tree wavelet transforms \cite{selesnick05}.

  We tested the  invariance property under translations and rotations of our randomized wavelet transform algorithm, when the results are averaged over multiple realizations, as follows.  We took an image on a 64$\times$64 square domain, and we considered a 5-level randomized wavelet decomposition. 
 We chose a processing task of projecting an image  onto the spaces $V_0, W_0, W_1, \cdots W_5$, which are described in Subsection \ref{subsec:multiresolution} and Subsection \ref{subsec:wavelets}.  This corresponds to
  transforming the image, and reconstructing from only a single level of
  coefficients.  We took averages of the results according to  100 different 
  wavelet realizations, and found that the realization-averaged transform is (almost) invariant under translations and rotations. This is illustrated in   The result is shown in Figure \ref{fig:translation_invariance}, for rotations (which is the harder of the two tasks); and it is
  seen that the process commutes with a rotation of 45 degrees, as desired. 

 \begin{figure}
  \centering
  \includegraphics[width=10.5cm]{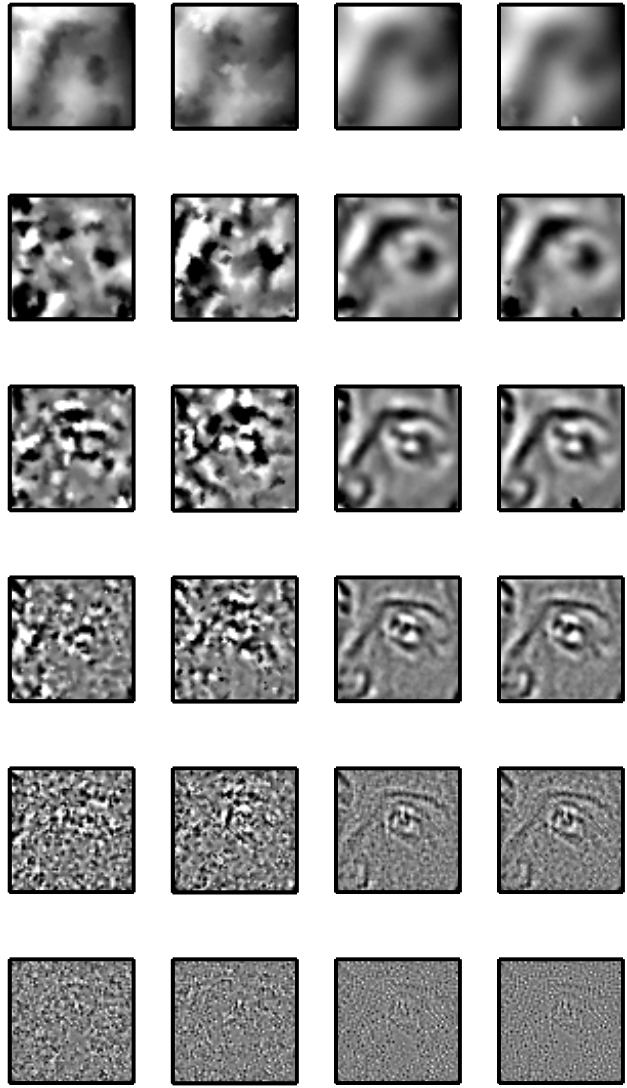}
  \caption{An image transformed and reconstructed back form a single level of wavelet coefficients, i.e, projected on to the corresponding subspaces for a 5-level wavelet decomposition. 
  The six rows correspond to the projection onto the spaces $V_0, W_0, W_1, \cdots W_5$, respectively. In the first and third 
  columns, images are processed without being translated whereas in the second and fourth columns they are first rotated 45 degrees, and then processed, and rotated back. The left two columns show the result of a single realization, and the right two columns show the result of computing the average of 100 realizations.  We see that the processing
  becomes much more powerful, and almost rotationally invariant when averaged over multiple realizations.}
  \label{fig:translation_invariance}
  \end{figure}
\chapter{Application to wavelet-based statistical analysis of fMRI data}
One of the most widely used and recognized methods for fMRI analysis  is
Statistical parametric mapping (SPM) \cite{hbf2}. In SPM, a key step is
spatial prefiltering with a Gaussian window, as a means of reducing noise. However, Gaussian filtering has the drawback of destroying the fine spatial details.   A wavelet-based alternative of SPM, which is called WSPM,  is due to Van De Ville et al. \cite{surfing, vandeville04,
vandeville0406, vandeville07}. It replaces the Gaussian filtering step with
wavelet filtering, and it involves thresholding in the wavelet domain as a denoising step, followed 
by a thresholding in the spatial domain. The threshold parameters in WSPM are selected so as to control the false positive rate, while minimizing the reconstruction error.   

Standard wavelets used in this framework
are defined on rectangular domains, typically a square (in two dimensions) or a cube (in
three dimensions).  On the other hand, the natural domain of the neural activity
is the brain cortex, which is an intricately convoluted three dimensional
domain.  

This chapter essentially  consist of the application of 
the domain-adapted wavelets of the previous chapter to the WSPM framework. The wavelets are constructed so as to have the brain cortex as their
natural domain.
\section{Statistical testing of fMRI data}
\subsection{One sample t-test}
\label{one_sample_t_test}
One of the most basic tasks in fMRI data analysis is to determine whether there is any brain response to a given stimulus or a given train of stimuli, and if there is, to determine its location.  In a simplified paradigm, we may assume the measurement we get from a given voxel is a random variable of the form 
$$
  v_n = \mu_n + e_n,
$$
where $e_n$ is a noise random variable with zero mean and unknown variance, and $\mu_n$ is the true activation in the voxel $n$, i.e. $\mu_n = 0$ when there is no activation and $\mu_n>0$ when there is activation. For each voxel $n \in V \subset \mathbb Z^3$  in the region of interest, one may consider using a standard one-sample t-test to decide whether that voxel is active or not.  Given a series of measurements  $v_n(1), v_n(2) \cdots, v_n(N)$ from voxel $n$,  the goal is to decide whether $\mu_n$ is zero or strictly positive, i.e., to decide between the hypotheses
\begin{align*}
&\mathbf{H_0} \colon \mu_n = 0 \\ 
&\mathbf{H_1} \colon \mu_n > 0.
\end{align*}
To decide which of the hypotheses is true, one passes to another random variable, which is called as a $t$-statistic, and distributed with Student's $t$-distribution of $N$ degrees of freedom. In order to obtain a $t$-variable,  the population mean $\hat \mu_n$ and an unbiased estimate of the variance $s^2_N$ are computed from the samples as follows:
$$
  \hat \mu_n = \frac{1}{N}\sum_{k = 1}^N v_n(k)\hspace{2cm} s^2_N=\frac{1}{N-1}\sum_{k=1}^{N}\left(v_n(k) - \hat \mu_n\right)^2.
$$
Then the ratio $t = \displaystyle \frac{\hat \mu}{\sqrt{s^2_N/N}}$ has a Student's $t$-distribution with $N$ degrees of freedom. A preselected extremely unlikely region is used to reject the null hypothesis.  In our case one would use the so called one-sided tail as in Figure \ref{ttest}, since the mean is assumed to be nonnegative. This region is chosen to have a small probability of occurrence, a typical value for the significance level $\alpha$ being $0.05$, which is also the false-positive rate, i.e., the chance that null hypothesis will be rejected incorrectly. 

\begin{figure}
\begin{center}
\subfigure[]{{\includegraphics[width=7.5cm]{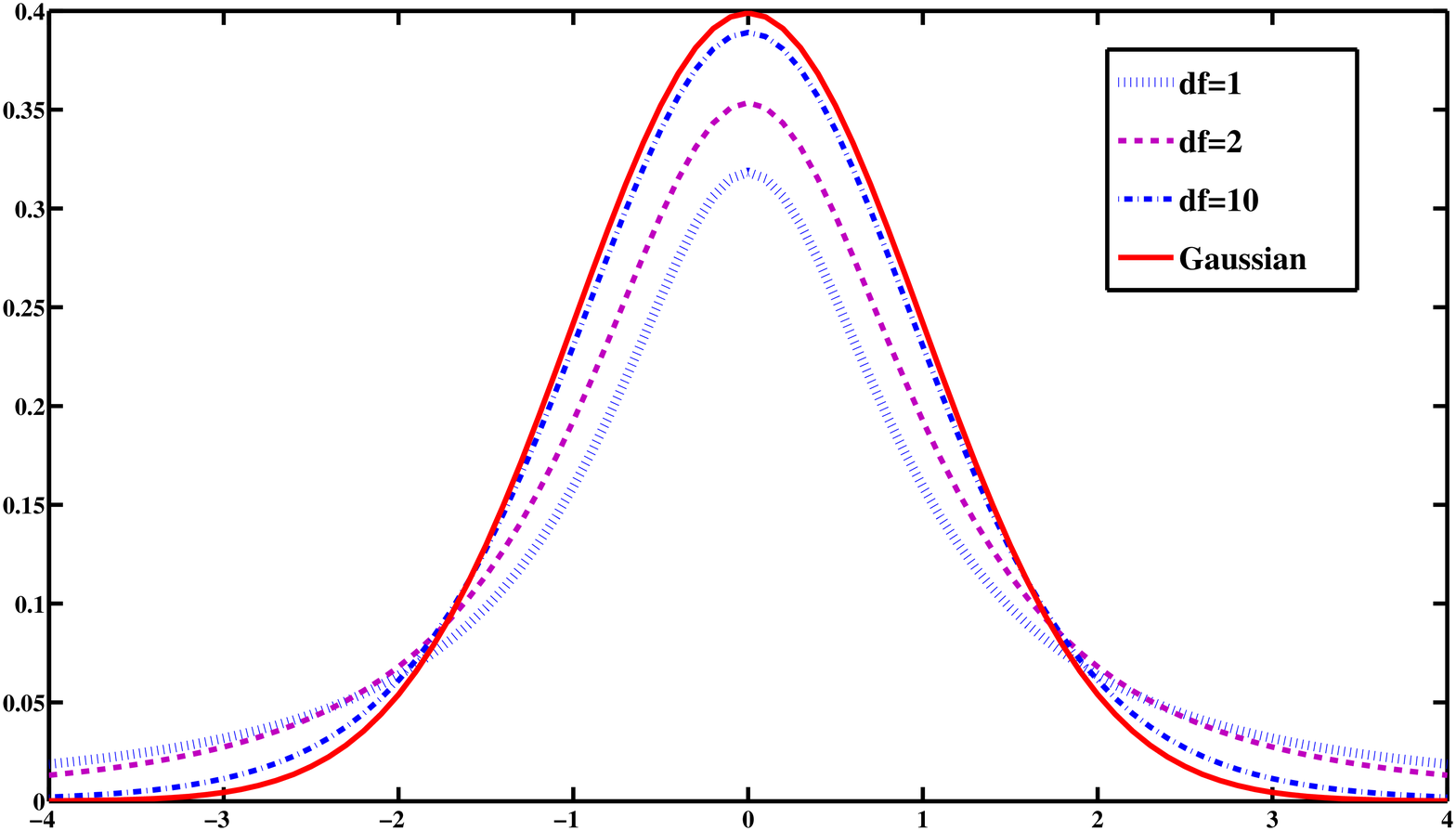}}}
\subfigure[]{{\includegraphics[width=7.5cm]{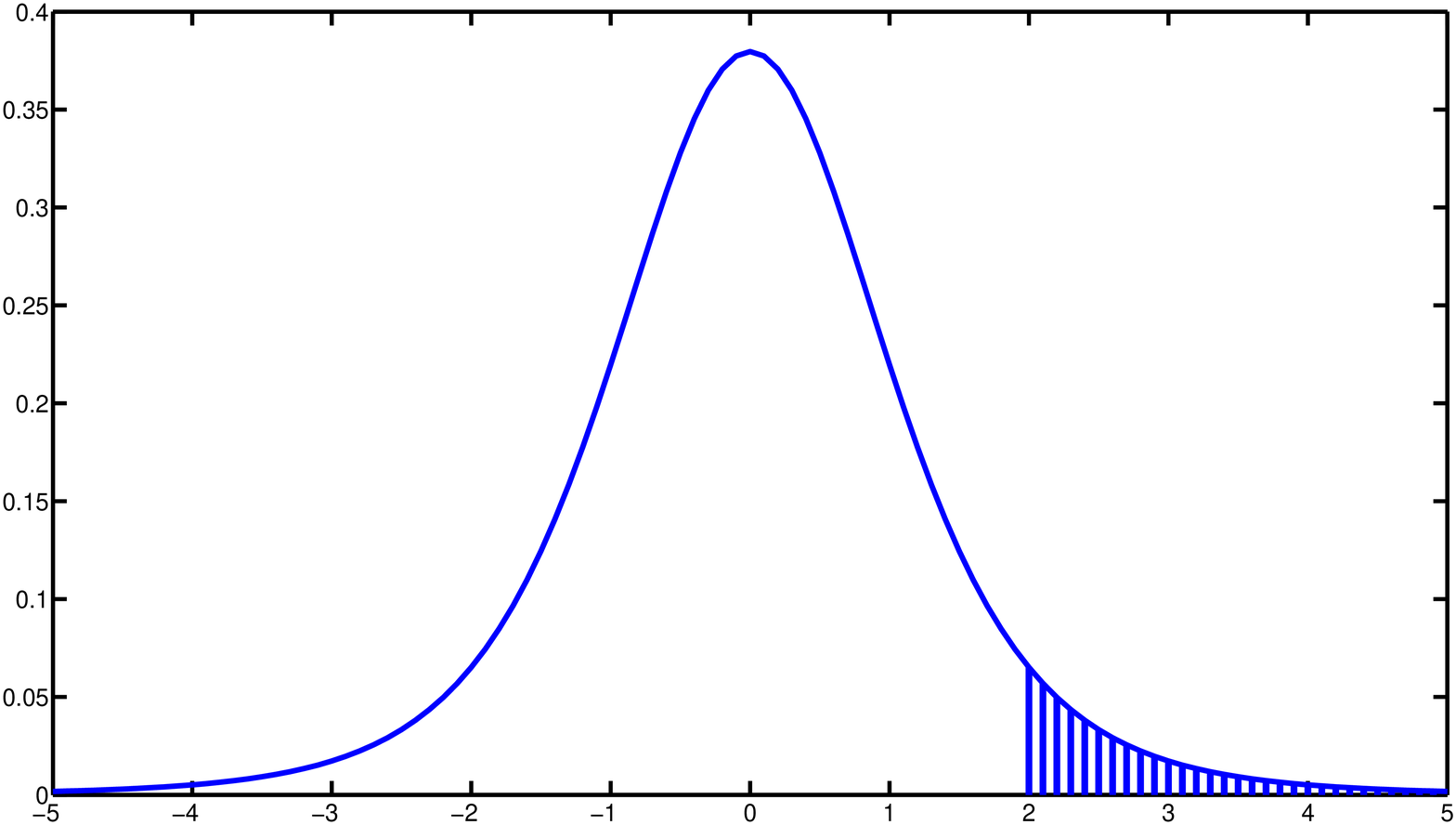}}}
\end{center}
\caption{(a) The density functions of Student's $t$-distribution with $n=1, 2$ and 10 degrees of freedom, respectively, and the standard Gaussian distribution, which can also be viewed as a $t$-distribution with infinitely many degrees of freedom. (b) The 5 \% region to reject the null hypothesis, for number of degrees of freedom $n=5$. }
\label{ttest}
\end{figure}

\subsection{The general linear model}
Although the model that is described in subsection \ref{one_sample_t_test} may be helpful for illustration purposes,  it was not realistic, since data from a real voxel is never silent, due to the highly active nature of the human brain. That is why one can observe only a slight increase in the activity of a voxel in response to stimuli.  In Figure \ref{fig_patterns_and_regressors} data from a real experiment by Haxby et al. \cite{haxby} are plotted, in which signals from 577 voxels are averaged for the correlation to be visually detectable.  In this setting, the question becomes whether there is a positive correlation between the stimulus and the response from the voxel and the decision is less straightforward. Fortunately, there exist straightforward extensions of the one sample $t$-test into this setting. 

\begin{figure}
 \centering
 \includegraphics[width=14cm]{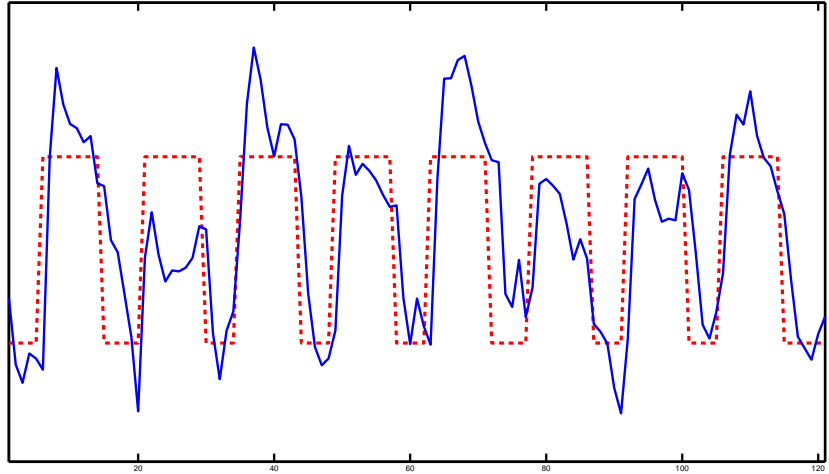}
 \caption{The stimuli and fMRI response.  Data from a visual stimuli experiment by Haxby et.al \cite{haxby}. The dashed line represents the on/off timing of the visual stimuli, and the solid line is the fMRI response, averaged over 577 voxels from the VT cortex.  The linear trend in the  stimuli is also removed. The correlation between the stimuli and the response is clearly visible in the figure, although in the signal from a single voxel the correlation would be virtually impossible to see.  }
 \label{fig_patterns_and_regressors}
\end{figure}

Let $v_n(1), v_n(2), \cdots, v_n(N)$ be the time samples obtained from the voxel $n$ at times $t = 1,2,\cdots, N$, and let us represent these samples as an $N\times 1$ column vector,
$$
\boldsymbol v_n = \left[ \begin{array}{c} v_n(1) \\ v_n(2) \\ \vdots\\ v_n(N) \end{array} \right]. 
$$

In a typical setting the subject is presented
with stimuli during certain intervals for a certain amount of time.
Let the voxel location $n$ be fixed, and
consider the time series $v(t)$, denoted as simply $v$.  We set up a model,
known as the \emph{general linear model}, or GLM, to analyze the time course of
this voxel:
\begin{equation*}
  \mbf{v}_n = \mbf X \boldsymbol{\beta} + \mbf e_n
\end{equation*}
where $\mbf X$ is a matrix each column of which is the expected time course if the subject
had responded only to a particular type of stimulus, with different stimulus type time courses given by the different columns of $\mbf X$. After the linear regression,
each component of the estimate of  $\beta$ will correspond to the level of
response to the corresponding category. 

A simple example for a single, on-off type stimulus with only eight time points can be given as
follows: Let us present the subject with stimuli for two time points, and then
give a rest for the next two time points, and repeat this procedure one more time.  Then a
simple choice for our model would be 
\begin{equation}
 \label{glm_example}
\mathbf v_n =  
 \left[ \begin {array}{rrr}
1&t_0&1\\\noalign{\smallskip}1&t_1&1\\\noalign{\smallskip}1&t_2&-1\\\noalign{\smallskip}
1&t_3&-1\\\noalign{\smallskip}1&t_4&1\\\noalign{\smallskip}1&t_5&1\\\noalign{\smallskip}
1&t_6&-1\\\noalign{\smallskip}1&t_7&-1\end {array} \right]  
\left[ \begin {array}{c}
{\beta_{0,n}}\\\noalign{\medskip}{\beta_{1,n}}\\\noalign{\medskip}{\beta_{2,n}}\end {array} \right]  
+ \mathbf e
.
\end{equation}
The constant vector in the first column accounts for the mean.  The second
column is the time vector, and its coefficient $\beta_{1,n}$ accounts for the linear trend.  The last column corresponds to the actual stimulus, having value 1 when the stimulus is on, and $-1$ when it is off.   The only parameter of interest here would be
$\beta_{2,n}$, corresponding to this last column, which would be interpreted as the magnitude of response from the current
voxel to the stimulus train.  

Then this analysis is repeated for each voxel, and the resulting map 
$\left\{\boldsymbol{\beta}_{2,n}\colon n\in R \right\}$  is then the activation parameter map for the given stimulus, where $R$ is the region of interest.  Significance of the response can be determined with the $t$-test, and the voxels passing the test are marked to be active. This is illustrated with sample time courses from a real experiment, in Figure \ref{glm_illustration}.

\begin{remark}
A more advanced approach would be to convolve the box functions in the columns of interest in  $\mathbf X$ with what is called as the \emph{hemodynamic response function}, which
makes it resemble the actual expected time course more, taking account the
dynamics of the blood-oxygenation in the brain.  See Figure {\rm\ref{hemodynamic_response}} for the plot of the convolved regressors and the hemodynamic response function, and \rm{\cite{sarty}} for more information. 
\end{remark}
\begin{figure}[h]
 \centering
 \includegraphics[width=12cm]{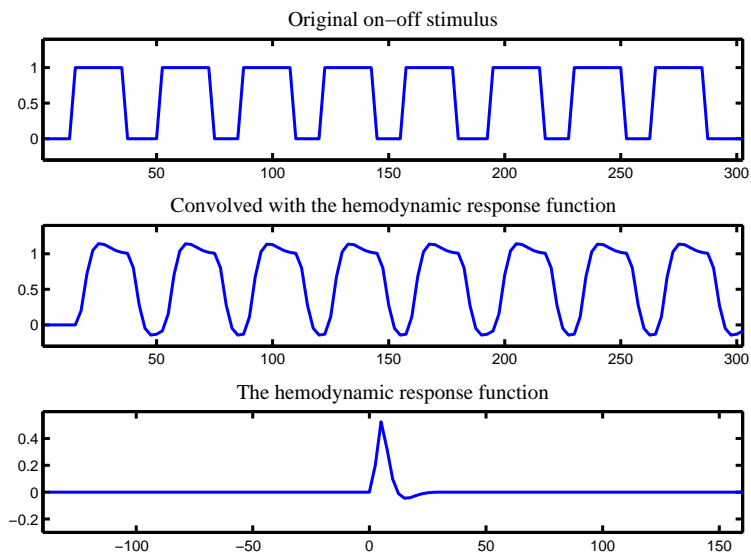}
 \caption{Original on-off stimulus (top),  convolved with the homodynamic response function (middle), and the hemodynamic response function (bottom).}
 \label{hemodynamic_response}
\end{figure}

\begin{figure}
 \centering
 \includegraphics[width=15cm]{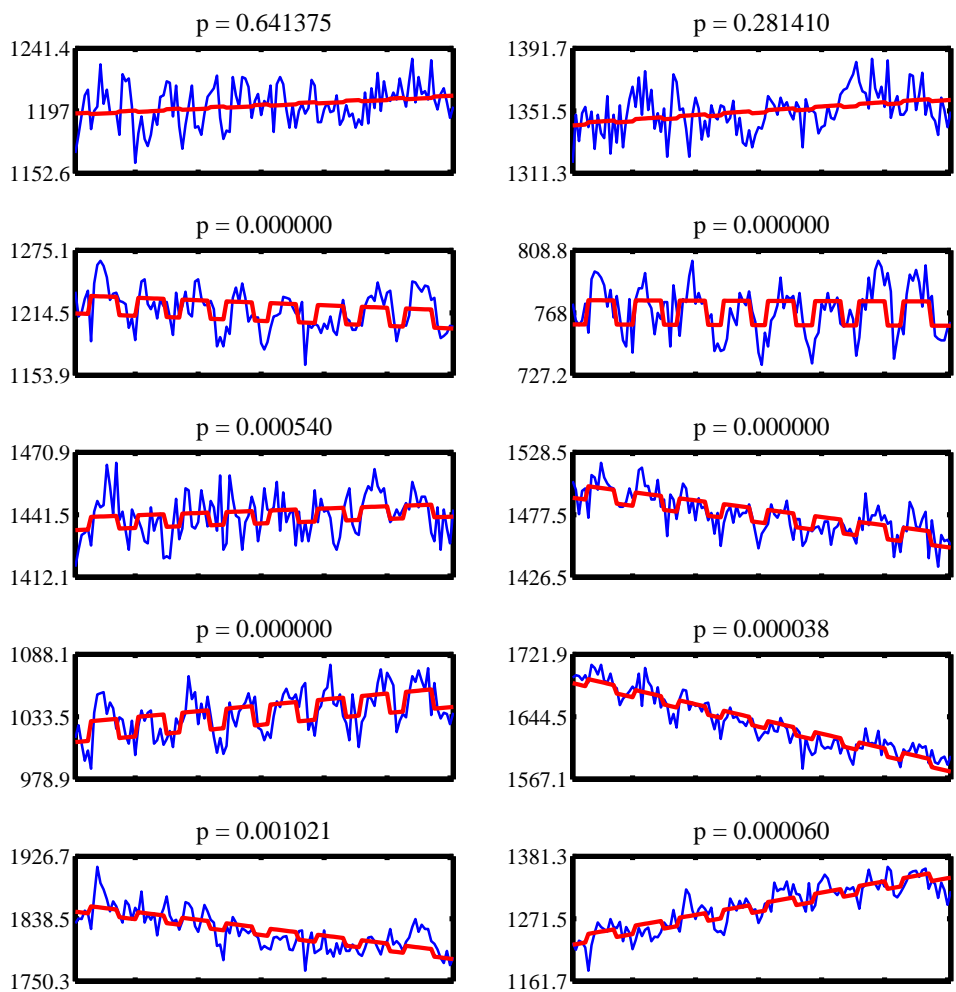}
 \caption{Plots of time series from various voxels, from the experiment by Haxby et al. \cite{haxby}.  Each voxel is fitted to the model, as a linear combination of a straight line and the box train function corresponding to the on-off stimuli.  A $t$ variable is derived and corresponding $p$ values are computed.  The smaller values of $p$ imply a stronger rejection of the null hypothesis.  }
 \label{glm_illustration}
\end{figure}

\subsubsection{Statistical testing}
The general linear model (GLM) for the time course of the voxel $n$ is
\begin{equation*}
  \mbf{v}_n = \mbf X \boldsymbol{\beta} + \mbf e, 
\end{equation*}
where  $\boldsymbol{\beta}$ is an $N\times 1$ vector of unknown parameters, and $\mathbf e$ is an $N\times 1$  Gaussian random vector with zero mean and unknown variance. The matrix $\mathbf X$ is $N\times L$, which is known beforehand and is called the \emph{design matrix}.  

Usually not all components of $\boldsymbol{\beta}$ are of interest.  One usually needs to know whether the quantity $\mbf{c}^{\operatorname{T}}\boldsymbol{\beta}$ is zero or strictly positive, and $\mbf c$ is called the \emph{contrast vector}. For the example of (\ref{glm_example}), the contrast vector would be taken as $\mbf{c}^{\operatorname{T}} = [0\ 0\ 1]^{\operatorname{T}} $, indicating that one is interested in whether the last column of the design matrix has positive correlation with the data, after removing the effects of the first two columns.   In another case when one is interested in whether column $i$ of the design matrix has more correlation with the data than column $j$, the contrast vector $\mbf c$ would be a vector of all zeros except for the $i$th and $j$th entries, which would be $1$ and $-1$, respectively. 

We observe a realization of the random vector $\mbf v_n$, and  would like to decide between the two hypothesis about $\mbf{\boldsymbol{\beta}}$:
\begin{align*}
&\mathbf{H_0} \colon \mbf{c}^{\operatorname{T}}\boldsymbol{\beta} = 0 \\ 
&\mathbf{H_1} \colon \mbf{c}^{\operatorname{T}}\boldsymbol{\beta} > 0.
\end{align*}
In order to do a $t$-test, we first need to obtain the $t$-variable. One starts with computing an unbiased estimate for $\boldsymbol{\beta}$ by an ordinary least squares formula
\[
\hat{\boldsymbol{\beta}} = (\tran{\mbf{X}} \mbf{X})^{-1}\mbf{X} \mbf{y},
\]
implicitly assuming the columns of $\mbf X$ are linearly independent. Then an estimate for the error would be 
\[
\hat{\mbf{e}} = \mbf{y} - \mbf{X}\hat{\boldsymbol{\beta}}.
\]
Now if one defines
 \begin{align}
  g_n & = \tran{\mbf{c}}\hat{\boldsymbol{\beta}},\\
  s^2_n  &= {\hat{\mathbf e_n}}^{\operatorname T}{\hat{\mathbf
e_n}}\ \mathbf c ^{\operatorname T}(\mathbf X^{\operatorname T} \mathbf
X)^{-1} \mathbf c. 
 \end{align}
 then the scalar random variable $g$ would have a Gaussian distribution with mean $\tran{\mbf{c}}{\boldsymbol{\beta}}$,   and $s^2$ will follow a Chi-square distributions with $J = N - \operatorname{rank}(X)$ degrees of freedom, and they are independent \cite{vandeville0406}.  Following this, one can obtain a $t$ variable by
$$
t_n = \frac{g_n}{\sqrt{s^2_n/J}}, \ \ \text{with} \ J =
N - \operatorname{rank}(\mathbf X),
$$ 
and one can then proceed similarly to the $t$-test of Subsection \ref{one_sample_t_test} for a statistical test of the significance of the activation.

\subsection{Multiple comparison problem}
When multiple statistical tests are to be performed based on the same data set, the problem of \emph{multiple comparison} needs to be dealt with.  In the case of fMRI, a typical dataset has of the order of $10^4$ voxels.  A significance level of $\alpha =  0.05$ would result in  $10^2$ to $10^3$ false detections; this number may be as large as the total number of voxels in the entire region of activation! In general, for a dataset of $K$ voxels, one expects to get $\alpha K$ false positives.   One solution offered to remedy this problem is to use the \emph{Bonferroni correction}, which is a conservative method to reduce the expected number of total false positives by reducing $\alpha$, and considering voxels in (sub)collections rather than individually.  For a total number of $K$ voxels, the Bonferroni correction asks to reduce the significance level from $\alpha$ to $\alpha/K$, which in turn reduces the expected number of total false positives from $\alpha K$ to $\alpha$.  By the \emph{union bound}, getting a single falsely detected  voxel within the set of $K$ voxels would have a rate smaller than $\alpha$. For a set with, e.g., $10^2$ voxels and an initial $\alpha$ selected to be 0.05, Bonferroni correction would ask to reduce $\alpha$ from $5\times 10^{-2}$ to $5\times10^{-4}$. Although it can be improved by restricting the region of interest to a smaller set of voxels,  this correction has the obvious drawback of extremely reduced sensitivity for fMRI datasets, almost to the level of detecting no activation \cite{surfing}: although the Bonferroni argument reduces the number of false positives, it \emph{increases} the number of false negatives. Such small choices of sensitivity rates make the tests useless, because fMRI data sets are very noisy and there are not enough samples to get detections at such a conservative specificity. 

\section{Spatial filtering for improving statistical power}
\subsection{Statistical parametric mapping and wavelets}

The problem with the Bonferroni correction is that it does not make use of the spatial correlation between voxels \cite{surfing}. There are methods that remedy this situation by essentially performing a transformation, which practically maps the large number of noisy and highly correlated voxels into a small number of uncorrelated and less noisy transform domain coefficients.  The most widely used approach in this category is the Statistical parametric mapping (SPM) package \cite{SPM}, in which the essential step is prefiltering by a Gaussian window, which corresponds to going to the Fourier domain, and throwing out the high frequency components and reconstructing.  Note that a spatial convolution with a Gaussian window is equivalent to multiplying the Fourier transform of the data with the Fourier transform of the Gaussian window; since this is another Gaussian, this multiplication is a form of weighted thresholding.  As depicted in Figure \ref{squarewave}, Gaussian and wavelet filtering can be viewed as two different forms of the same idea. In Gaussian filtering, the coefficients to be discarded are the high frequency coefficients that are predetermined, which makes it linear, while in wavelet filtering, coefficients are to be discarded are determined by thresholding, therefore it is a form of nonlinear filtering. 

\begin{figure}
\begin{center}
\subfigure[]{{\includegraphics[width=7.5cm]{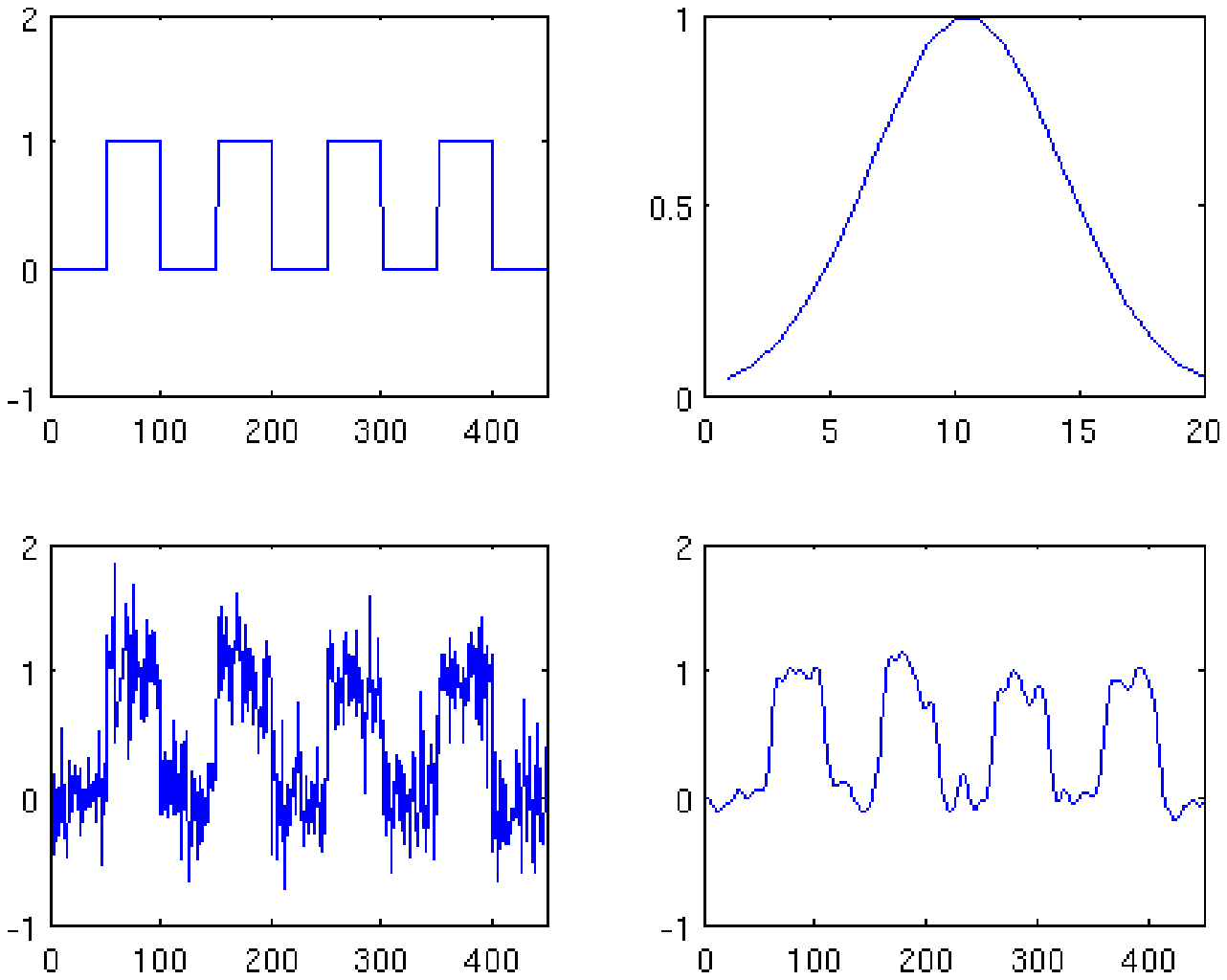}}}
\subfigure[]{{\includegraphics[width=7.5cm]{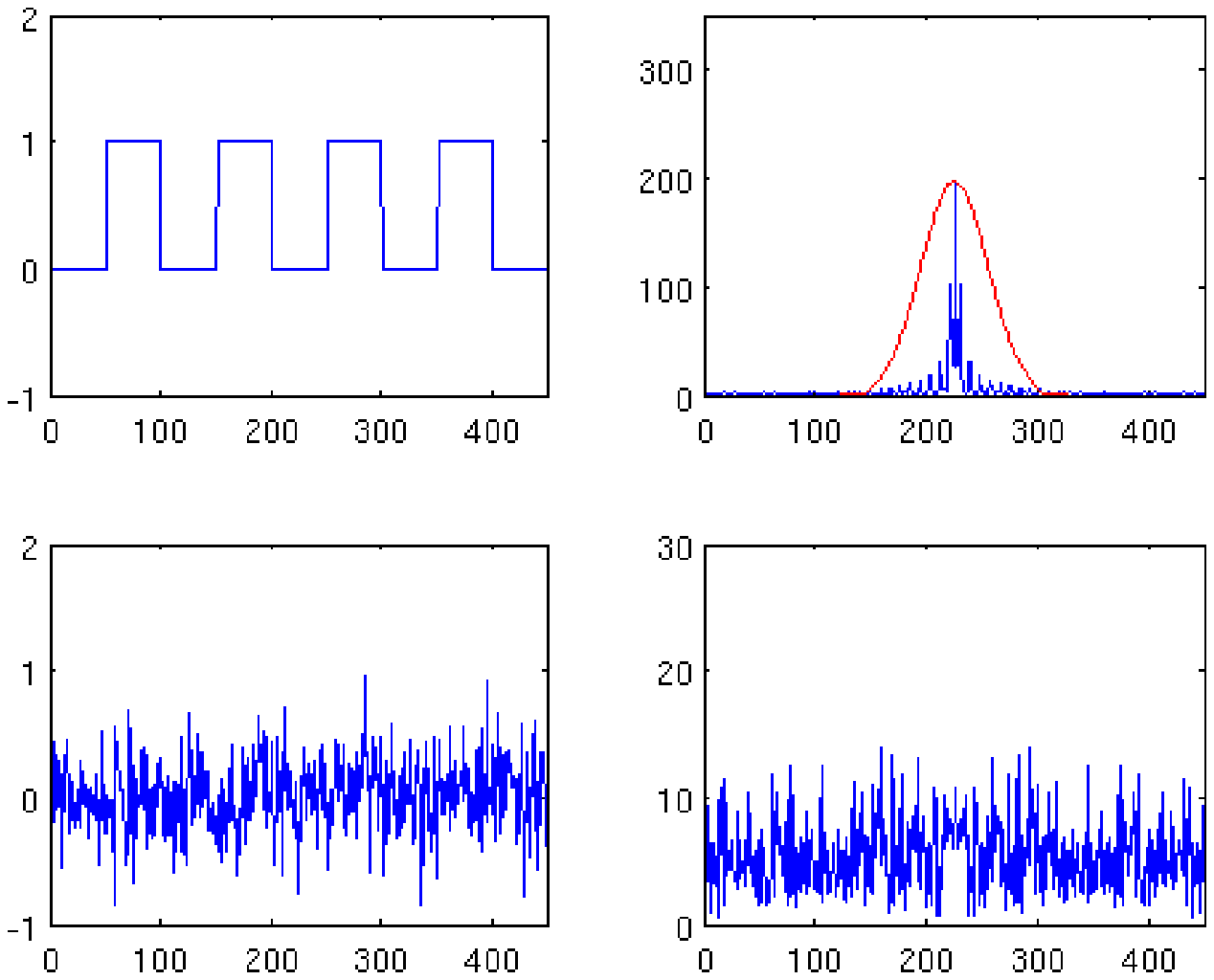}}}
\subfigure[]{{\includegraphics[width=7.5cm]{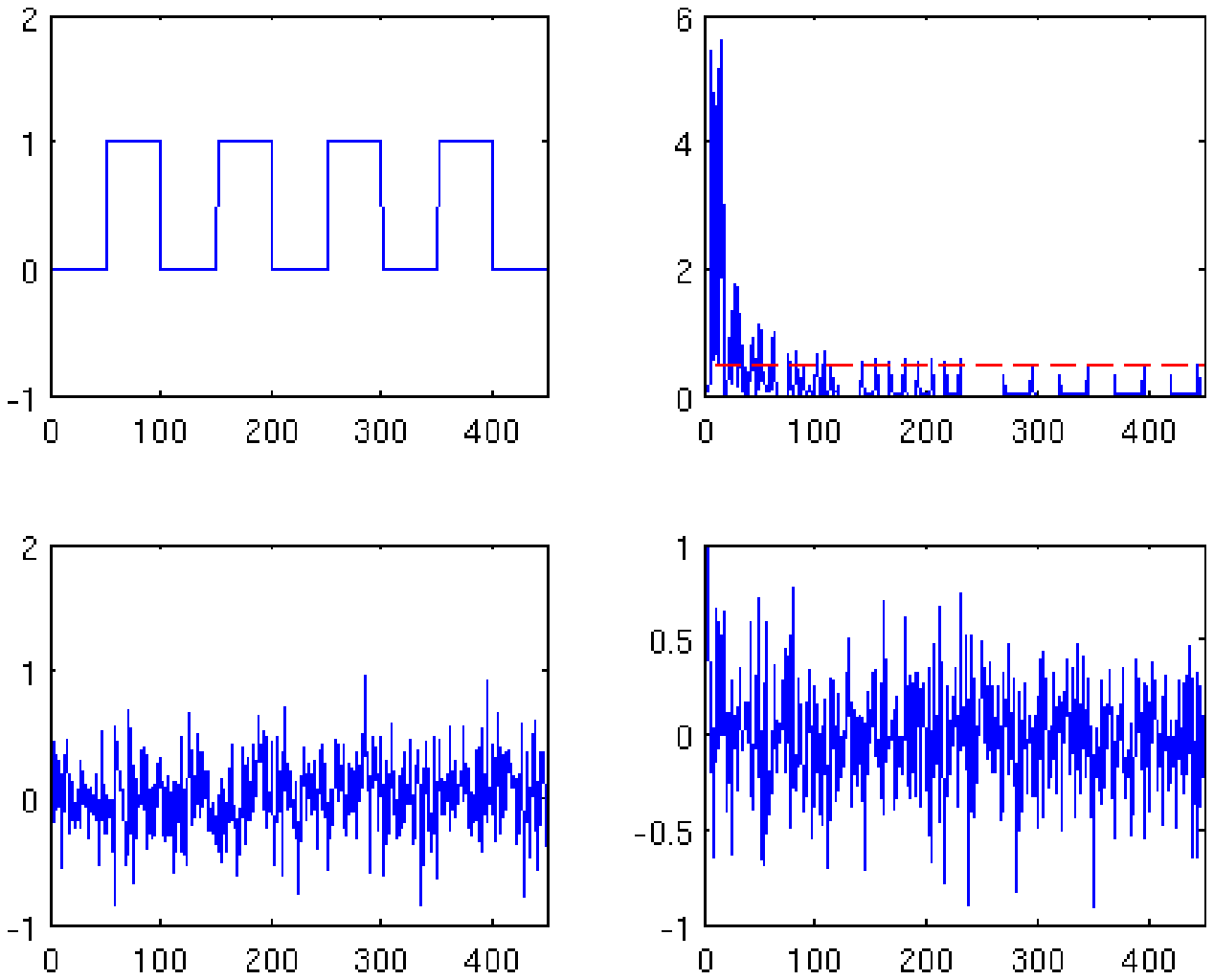}}}
\end{center}
\caption{Depiction of noise reduction by Gaussian and wavelet filtering.  Spatially correlated functions have their transform domain concentrated to a few coefficients, while noise processes have their transforms spread out to all coefficients uniformly. (a) First Row: A square wave function, and a Gaussian which will be used in filtering, Second Row: Noisy (left) and filtered (right) versions of the square wave function.  (b) First row: The square wave function and its Fourier transform, overlaid by the Fourier transform of the Gaussian window, Second Row: White noise and its Fourier transform. (c) First Row: The square wave function and its Wavelet transform, Second Row: White noise and its wavelet transform.}
\label{squarewave}
\end{figure}

\label{sec:theory}

\subsection{Gaussian versus wavelet filtering }
While Gaussian filtering has the advantages of being a well established linear noise reduction method, its main disadvantage is that it destroys all fine spatial details, which might be important in fMRI images. In Figure \ref{wavelet_filtering}, this phenomenon is illustrated with a single dimensional example.  That is the main motivation of exploring wavelet based methods, which are the subject of the rest of this chapter. 

\begin{figure}
 \centering
 \includegraphics[width = 15cm]{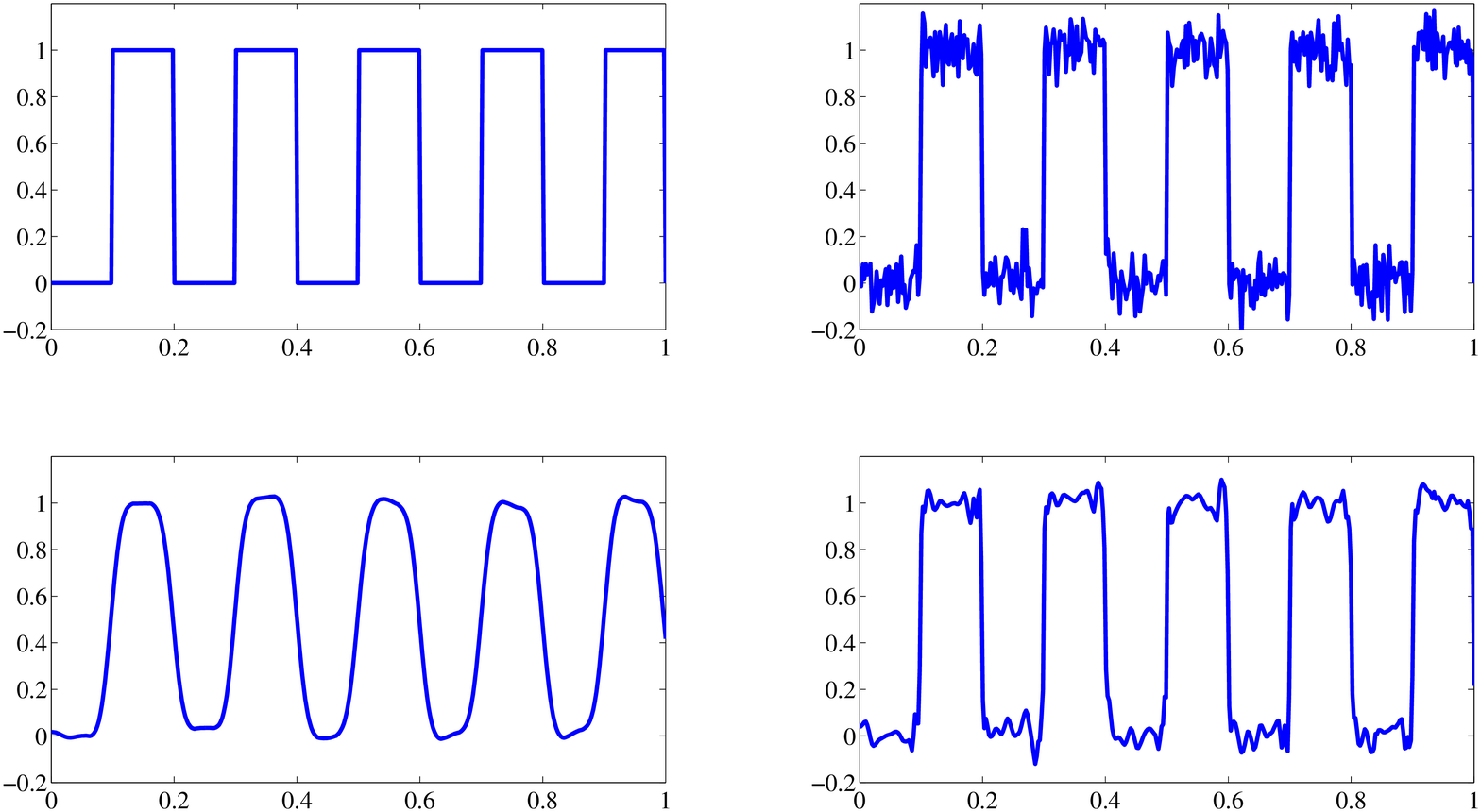}
 \caption{Wavelet denoising. A square wave signal is contaminated with noise. In the second row, it is denoised by Gaussian filtering and by wavelets, respectively.  Although Gaussian filtering achieves a visually detectable noise reduction, it also destroys the fine spatial details, which corresponds, in this example, to the steepness of the edges; this steepness is much better preserved by wavelet denoising.}
 \label{wavelet_filtering}
\end{figure}

\subsection{Classical wavelet-based analysis}
The classical wavelet-based methods \cite{ruttimann9801}, mainly transfer each spatial volume to the wavelet domain, perform a $t$-test on each wavelet coefficient, and reconstruct the volumes back by only using the wavelets whose coefficients survive the $t$-test. 

Let us denote an fMRI
data set with $v_{\mathbf n}(t)$, where $\mathbf n \in \mathbb Z^3$ is the
spatial index, and $t\in\mathbb Z$ is the temporal index.  With a simplified
notation, the wavelet transform corresponds to representing the data under the form
$$
v_\mathbf n (t) = \sum_\mathbf k w_\mathbf k(t) \psi_\mathbf k (\mathbf n),
$$
where $\{\psi_\mathbf k\}$ is the wavelet basis.  The basis functions
are translated and dilated versions of some prototype in the case of standard wavelets, but they take more arbitrary forms in the adapted case.

Now let $\mathbf w _\mathbf k$ be the vector of wavelet coefficients
corresponding to $\psi_k$; i.e., 
$$\mathbf w_ \mathbf k = [w_\mathbf k (1)\ 
 \cdots\  w_\mathbf k (N_t)]^T,$$ where $N_t$
is the total number of time samples.  Due to linearity of the wavelet transform, we can write the same general linear model
as in the spatial domain
$$
\mathbf w_\mathbf k = \mathbf X \mathbf y
_\mathbf k + \mathbf e_\mathbf k,
$$ 
where $\mathbf X$ is the $N_t \times L$
design matrix, $\mathbf y_\mathbf k$ is the $N \times 1$ vector of unknown
parameters, and $\mathbf e_\mathbf k$ is the residual error.   Assuming the noise to be independent and identically
distributed Gaussian, the unbiased estimate of $\mathbf y_\mathbf k$
is given by
$$
\hat{\mathbf y}_\mathbf k = (\mathbf X^{\operatorname T} \mathbf
X)^{-1}\mathbf X^{\operatorname T} \mathbf w_\mathbf k .$$  
Corresponding to each index $\mathbf k$ of the wavelet basis, we obtain two
scalar values:
\begin{align*}
g_\mathbf k &= {\mathbf c}^{\operatorname T} \hat {\mathbf y}_\mathbf k,\\
s^2_\mathbf k &= {\hat{\mathbf e}_\mathbf k}^{\operatorname T}{\hat{\mathbf
e}_\mathbf k} \mathbf c ^{\operatorname T}(\mathbf X^{\operatorname T} \mathbf
X)^{-1} \mathbf c,
\end{align*}
where $\mathbf g_\mathbf k$ and $\mathbf s_\mathbf k$ follow a Gaussian and
Chi-squared distribution, respectively, and $\mbf c$ is the contrast vector.  From these one can obtain a $t$ value
for each wavelet coefficient $\mathbf k$:
$$
t_\mathbf k = \frac{g_\mathbf k}{\sqrt{s^2_\mathbf k/J}}, \ \ \text{with} \ J =
N_t - \operatorname{rank}(\mathbf X),
$$ 
which can be tested against a threshold $\tau_w$, which is chosen in accordance
with the desired significance level.  After testing, the
detected coefficients are reconstructed as
\begin{equation}
\label{eq:recon}
\mathbf r_\mathbf n = \sum_\mathbf k T_{\tau_w}(t_\mathbf k)g_\mathbf k
\psi_\mathbf k(\mathbf n),
\end{equation}
where $T$ is the thresholding function corresponding to the two
sided $t$ test; i.e., $T(t_\mathbf k) = 1$ if $|t_\mathbf k|\geq \tau_w$, and
zero otherwise. The volume $r_\mathbf n$ contains many nonzero voxels, each of
which is a function of many voxels from the original data.  One must rely on heuristic thresholds on $r_\mathbf n$ to
obtain acceptable detection maps. Moreover, $r_\mathbf n$ does not have a direct
statistical interpretation.  These disadvantages are overcome with the WSPM, as explained in the next section.

\subsection{WSPM: Joint spatio-wavelet statistical analysis}

Wavelet based Statistical Parametric Mapping (WSPM) is a method proposed by Van De Ville et al. \cite{surfing, vandeville04,
vandeville0406, vandeville07}, which is a modification of SPM where the denoising step is performed by thresholding in the spatial wavelet domain.  This makes the typical advantage of wavelets, which is about reducing noise while keeping high frequency detail information, apparent in
the results. The underlying theorem guarantees control over the false-positive rate by a bound of the null-hypothesis rejection probability. Moreover, empirical results show similar sensitivity to that obtained by SPM with improved spatial detail.  The resulting map of active voxels can be seen to align with the
cortex, as a demonstration of preserving the detail information. 

The main idea in the joint spatio-wavelet statistical analysis is
to perform two consecutive thresholding operations: first in the wavelet domain
and then  in the spatial domain. There are two corresponding threshold
parameters, 
$\tau_w$ and
$\tau_s$, to be determined  \cite{vandeville04, vandeville0406}.  A spatial map
is obtained after the first thresholding as in (\ref{eq:recon}). Then $r_\mathbf
n$ is weighted by $1/\sum_\mathbf k \sigma_\mathbf k |\psi(\mathbf n)|$ and
thresholded by $\tau_s$.  That is, the set of voxels that are declared to be
active would be
$$
\left\{\mathbf n \colon \frac{\left|\sum_\mathbf k T_{\tau_w}(t_\mathbf
k)g_\mathbf k
\psi_\mathbf k(\mathbf n)  \right|}{\sum_\mathbf k \sigma_\mathbf k
|\psi(\mathbf n)|}\geq \tau_s\right\}.
$$

Given the desired significance level $\alpha$,  an optimal choice for $\tau_s$
and
$\tau_w$, which
minimizes
the approximation error between the reconstruction from the fitted parameters
and two times
thresholded reconstruction,  can be computed to be
$$
 \tau_w =
\sqrt{-W_{-1}\left(-\frac{\alpha^2\pi}{2}\right)},\ \ \ \tau_s = 1/\tau_w,
$$
where $W_{-1}$ is the $-1$-branch of the Lambert-$W$ function
\cite{vandeville04}. 

\begin{figure}
 \centering
 \includegraphics[width=14cm]{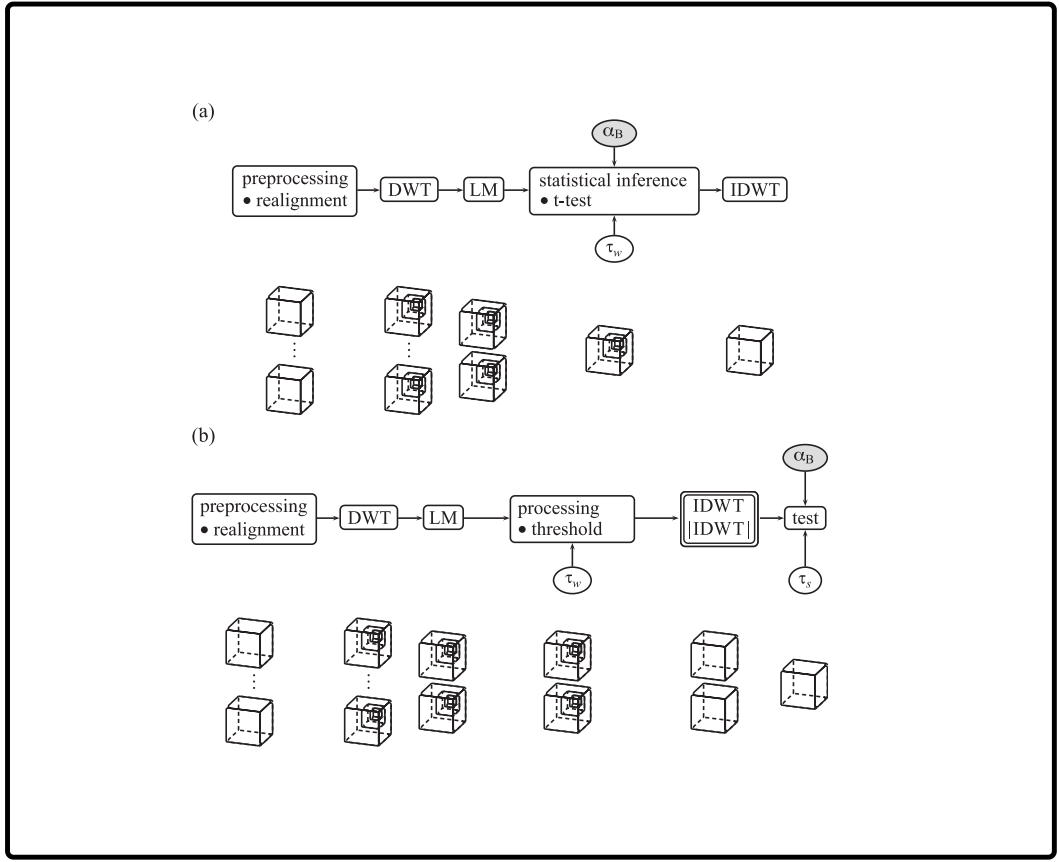}
 \caption{Diagram illustrating the wavelet-based approach proposed in \cite{vandeville0406}.  a) Classical wavelet based method, and b) Integrated wavelet based method (WSPM).  
 In both methods,  the wavalet transform of each volume is computed (DWT), the variables $g_\mathbf k$ and $s^2_\mathbf k$ are computed for each wavelet coefficient (LM), a $t$ variable is obtained for each wavelet coefficient and the $t$-test is performed, and the inverse wavelet transform computed after throwing out the coefficients failing the test (IDWT).  In the classical method, one must rely on heuristic thresholds for the reconstructed volume.  In the integrated framework, the thresholds $\tau_w$ and $\tau_s$ are determined together, as a function of the input sensitivity parameter $\alpha_B$. In the integrated method, one also needs to compute a reconstruction with the absolute value of the wavelets, before the spatial thresholding.     Figure replicated from \cite{vandeville0406}.}
 \label{fig:classical_vs_vspm}
\end{figure}

\subsection{Anatomically adapted wavelets}
In the methods mentioned above methods, the wavelet transforms are either performed slice by
slice as two dimensional wavelet transforms, or are applied to the whole
volume as a three dimensional wavelet transform.  In either case, the domain of
the signals are assumed to be the
rectangular. However, in reality the activity takes place in a subset
corresponding to the brain cortex, which is a highly convoluted three
dimensional structure. This motivates us to obtain domain-adaptive wavelets for fMRI data, using the construction in Chapter 3.   We input the segmented brain cortex  to the algorithm, and use the resulting wavelets in the statistical framework for analyzing fMRI data.  We refer to these type of wavelets as \emph{anatomically adapted wavelets}. 



%



%

\section{Experimental results}
\label{sec:expresults}
\subsection{Simulated data}
We generated smooth functions over a domain consisting of concentric
circles, to represent the fine layered structure of the gray-matter cortical layer with different widths. We contaminated the data with white Gaussian noise, whose magnitude is decreasing from left to right, as shown in Figure
\ref{fig:res2}. 
The adapted wavelets showed an improvement in the
sensitivity, while keeping the specificity within the theoretical limits.  As the level of the wavelet decomposition increases, the sensitivity keeps
increasing when adapted wavelets are used, while it remains unchanged for
standard wavelets; this is illustrated by the comparison in  Figure \ref{fig:res3} .
\begin{figure}[htb]

\begin{minipage}[b]{1.0\linewidth}
  \centering
  \centerline{\epsfig{figure=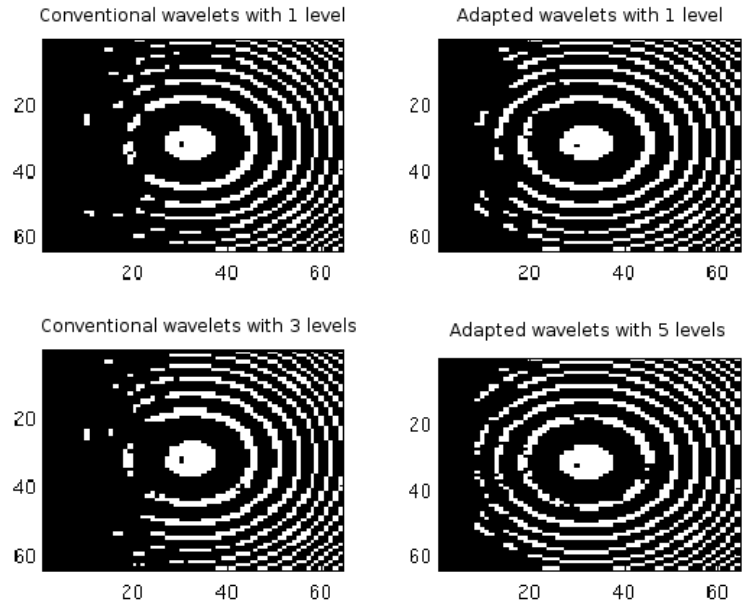,width=12cm}}
\medskip
\end{minipage}
%

\caption{The detected map of active voxels in an example with simulated data. The domain consists of concentric rings.  Gaussian white noise is superimposed on the data, with $\sigma$ kept constant within each ``column'', but as the ``row'' number $r$ increases, $\sigma$ decreases, with $\sigma r = \text{constant}$.  For $r=2$, $\sigma$ is such that the noise overpowers the signal completely; for $r=64$, $\sigma$ has decreased to a level where the signal is detected without problems.  The left column shows results with
standard orthogonal wavelets, and the right column shows the result with the 
anatomically adapted wavelets, which detect the signal in some locations where standard wavelets don't. In the second row the level of wavelet
decomposition is increased, and the performance of adapted wavelets has
increased correspondingly while the performance of standard wavelets remain unchanged.}
\label{fig:res2}
\end{figure}

\begin{figure}[htb]


\begin{minipage}[b]{1.0\linewidth}
  \centering
  \centerline{\epsfig{figure=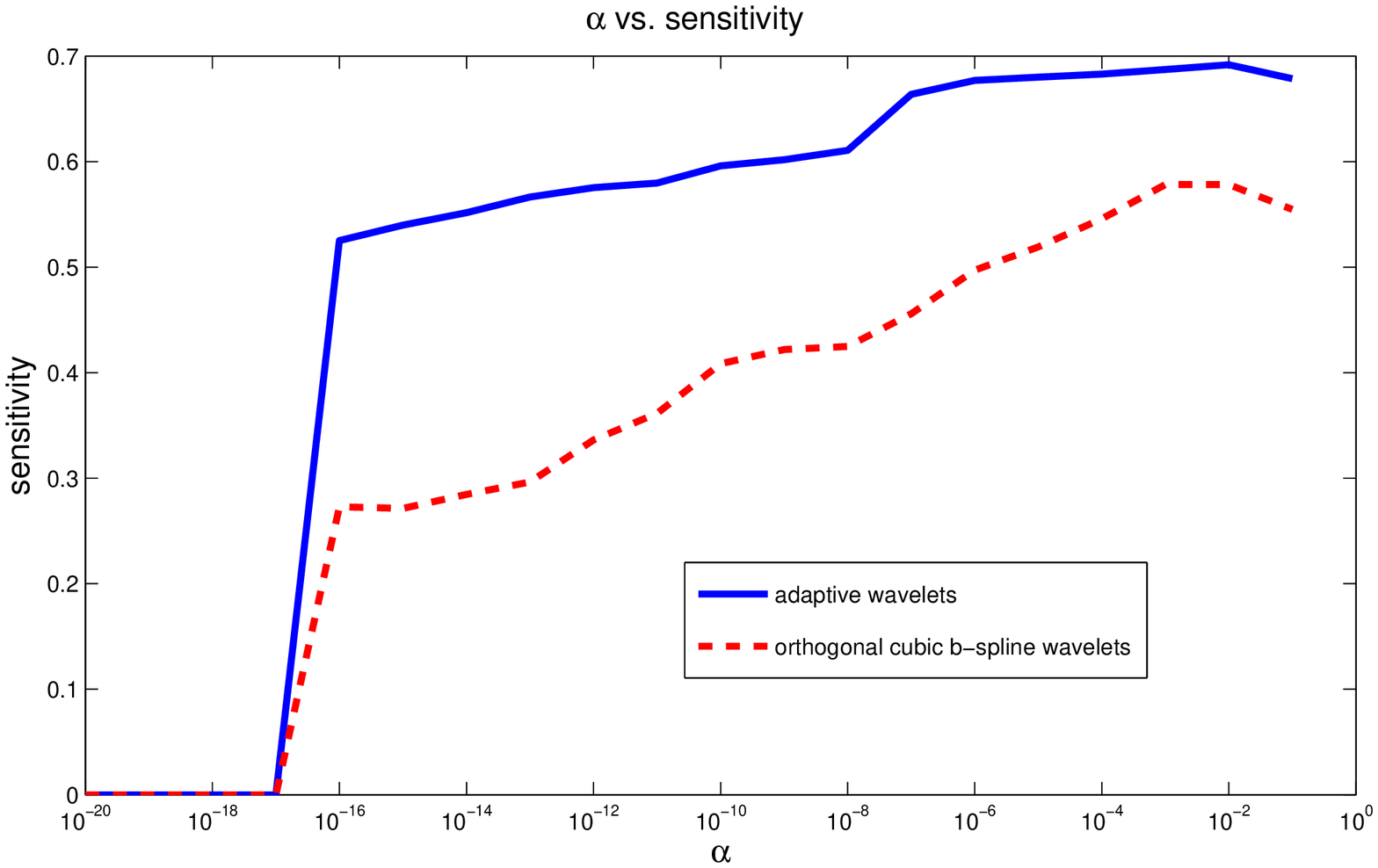,width=12cm}}
\end{minipage}
\caption{The ROC curves with standard (tensor-product orthogonal cubic B-spline
wavelets) and anatomically adapted wavelets. The $\alpha$-value represents the
type-I error control rate that is input to the algorithm.}
\label{fig:res3}
\end{figure}

\subsection{Real data}
We tested the proposed method with data obtained from a visual stimulation
experiment, with 16 slices of $128\times128$ voxels, the size of which is 1.8 mm $\times$ 1.8
mm $\times$ 5 mm. We performed segmentation with the SPM software package, and generated the adaptive
wavelets using the domain corresponding to the gray-matter layer. The thresholded binary domain is shown in Figure \ref{fig:cortical_mask}. Using the
standard
orthogonal wavelets, the analysis resulted in 1032 detected voxels,  with
the adapted wavelets it resulted in 1214 active voxels. In both cases the
sensitivity parameter $\alpha$ is taken to be 0.001. This suggests an
improved sensitivity, with a detection of larger number of voxels, as shown in
Figure \ref{fig:res4}. 
\begin{figure}
 \centering
 \includegraphics[width=14cm]{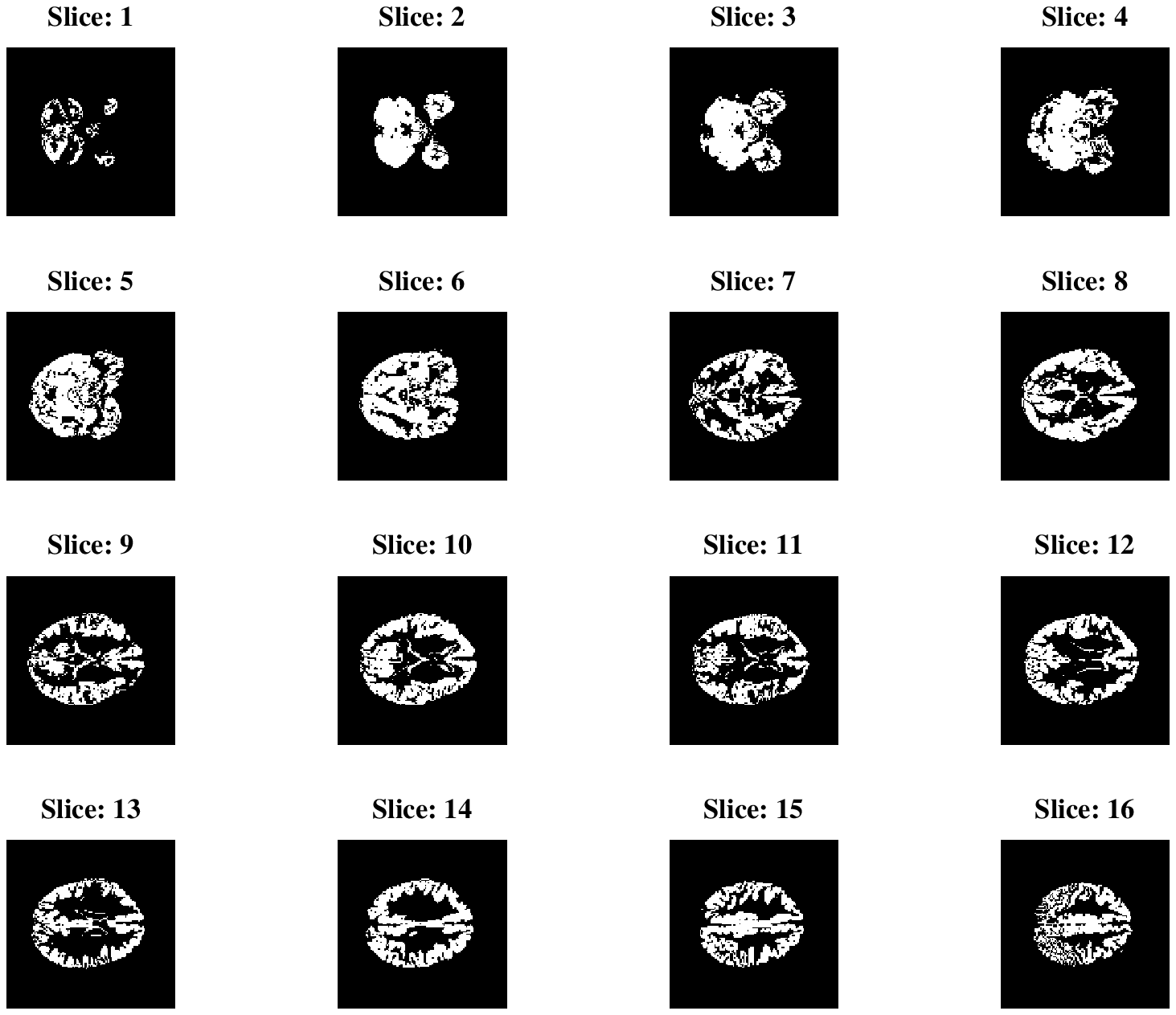}
 \caption{The binary mask used in the algorithm, which identifies the brain cortex.  The volume corresponding to the brain cortex is displayed slice by slice.  This map is obtained by thresholding the output of SPM's brain segmentation algorithm. This domain information is input to the wavelet construction algorithm, as a binary volume.  }
 \label{fig:cortical_mask}
\end{figure}

\begin{figure}[htb]
\begin{minipage}[b]{1\linewidth}
\centering
\centerline{\epsfig{figure=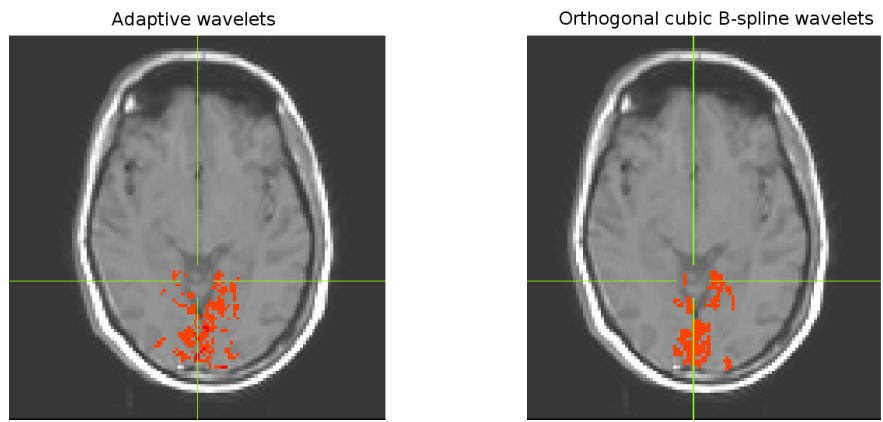,width=12cm}}
\end{minipage}
\caption{Slices from detected activation maps with real visual stimulation
data. The left image is obtained with adaptive wavelets, and the right image is obtained
with orthogonal cubic B-spline wavelets.}
\label{fig:res4}
\end{figure}

 \nocite{sweldens_siam}
\nocite{surfing} 
\nocite{vandeville0406}  \nocite{vandeville04} \nocite{sarty}
\nocite{daubechies}
\nocite{ruttimann9801}  \nocite{vandeville07}
\nocite{sweldens96}


\chapter{Summary and future work}
In this thesis, we constructed two and three dimensional wavelets for arbitrary domains, using the lifting scheme.  When the original domain of the signals to be analyzed has a significant proportion of boundary pixels or voxels,  our adapted wavelets demonstrate better sparsity properties and  a superior performance in denoising, compared to standard wavelets that have rectangular domains. 

In the construction, a nested grid structure is required to be defined on the domain. In defining the nested grid, we used a randomized algorithm.  This randomization gives us the chance to have multiple sets of wavelet bases on the same domain, which allows one to process the data multiple times and then take the average of the results.  This, in general,  improves the results whenever  it is possible to average the output of the wavelet application.  Our test with two dimensional images showed that this algorithm is nearly translation and rotation  invariant.

The new class of wavelets are then used in the brain imaging problem, after being adapted to the
anatomy of the brain cortex.  In the wavelet-based Statistical Parametric Mapping
framework,  we have observed an improved sensitivity, while
retaining the same amount of control over type-I errors, compared to wavelet transforms having rectangular domains. In simulated data, contrary
to what is observed with the standard wavelet transform, use of the adapted wavelets
shows clear improvement as the level of the wavelet decomposition increases.

As the high resolution fMRI scanners are becoming more widely available, spatial filtering tools like ours, which treat the brain cortex as an arbitrary three dimensional volume,  may become an alternative to the well-established spatial filtering tools that are currently used in brain imaging analysis.

\subsubsection{Future work}
In utilizing  multiple sets of wavelet bases, we processed the data separately with each individual non-redundant basis, and took the average of the results.  However, the average may include results from some bad realizations of the basis, which may deteriorate the performance. As an alternative, we will explore considering of the union of the multiply generated bases as a single overcomplete basis, and use it together with sparsity-constrained risk minimization algorithms. 

In another direction, we will test the performance of the anatomically adapted wavelets in multi-subject fMRI studies. 
\nocite{sweldens_siam} \nocite{unbalanced_haar}
\nocite{donoho} \nocite{donoho_transinv} \nocite{Sweldens_buildingyour}
\nocite{surfing} \nocite{sarty} \nocite{cohen_dahmen_devore} \nocite{sweldens95} \nocite{daubechies98} \nocite{fattal09} \nocite{coifman06} \nocite{cohen92} \nocite{dahmen97} \nocite{vasilyev00} \nocite{vandeville06} \nocite{ozkaya11}
\bibliographystyle{plain}
\bibliography{mybib}

\end{document}